\documentclass[11pt]{amsart}

\usepackage{amscd,amssymb,amsopn,amsmath,amsthm,mathrsfs,graphics,amsfonts,enumerate,verbatim,calc
}
\usepackage[dvips]{graphicx}
\usepackage{url}

\usepackage[OT2,OT1]{fontenc}
\newcommand\cyr{%
\renewcommand\rmdefault{wncyr}%
\renewcommand\sfdefault{wncyss}%
\renewcommand\encodingdefault{OT2}%
\normalfont
\selectfont}
\DeclareTextFontCommand{\textcyr}{\cyr} 

\usepackage{mathtools} %\coloneqq etc
\usepackage{colonequals} %\colonequals

\usepackage[utf8]{inputenc}
\usepackage[T1]{fontenc}
\usepackage{lmodern}
\usepackage[english]{babel}
\usepackage[autostyle]{csquotes}

\usepackage{tikz-cd}

\usepackage{xcolor}
\definecolor{darkgreen}{rgb}{0.0, 0.6, 0.0}
\usepackage[bookmarks=true,draft=false,breaklinks,colorlinks,citecolor=darkgreen,linkcolor=red]{hyperref} %白黒印刷するときはコメントアウト
\RequirePackage{cleveref} % It is convenient to use \Cref instead of \ref. Using this, we need not write Lemma, Theorem, Corollary, ... before \ref.

\usepackage{amssymb,amsmath}

\DeclareFontFamily{OT1}{rsfs}{}
\DeclareFontShape{OT1}{rsfs}{n}{it}{<-> rsfs10}{}
\DeclareMathAlphabet{\mathscr}{OT1}{rsfs}{n}{it}

\numberwithin{equation}{section}
\newtheorem{theorem}{Theorem}[section]
 \crefname{theorem}{Theorem}{Theorems}
\newtheorem{lemma}[theorem]{Lemma}
 \crefname{lemma}{Lemma}{Lemmas}
\newtheorem{proposition}[theorem]{Proposition}
 \crefname{proposition}{Proposition}{Propositions}
\newtheorem{proposition-definition}[theorem]{Proposition-Definition}
 \crefname{proposition-definition}{Proposition-Definition}{Proposition-Definition}
\newtheorem{corollary}[theorem]{Corollary}
 \crefname{corollary}{Corollary}{Corollaries}
\newtheorem{corollary-definition}[theorem]{Corollary-Definition}
 \crefname{corollary-definition}{Corollary-Definition}{Corollary-Definition}
\newtheorem{claim}[theorem]{Claim}

 \crefname{maintheorema}{Main Theorem A}{Main Theorem A}

\theoremstyle{definition}
\newtheorem{definition}[theorem]{Definition}
 \crefname{definition}{Definition}{Definitions}
\newtheorem{remark}[theorem]{Remark}
 \crefname{remark}{Remark}{Remarks}
\newtheorem{example}[theorem]{Example}

 \crefname{observation}{Observation}{Observations}
\newtheorem{notation}[theorem]{Notation}
 \crefname{notation}{Notation}{Notations}

\theoremstyle{remark}

\newtheorem{acknowledgement}{Acknowledgement}

\newenvironment{pfclaim}[1][$\Diamond$]
{\def\claimQED{{#1}}\noindent {\em Proof of the claim. }}
{\leavevmode\unskip\penalty9999 \hbox{}\nobreak\hfill
    \quad\hbox{\claimQED}{\smallskip}}

\newcommand{\rad}{\operatorname{rad}}

\newcommand{\Tor}{\operatorname{Tor}}

\newcommand{\Frac}{\operatorname{Frac}}

\newcommand{\pr}{\operatorname{pr}}

\newcommand{\mult}{\operatorname{mult}}

\newcommand{\Frob}{\operatorname{Frob}}

\usepackage{enumitem}
\newlist{enumalph}{enumerate}{1}
\setlist*[enumalph]{label={\upshape(\alph*)}, nosep}

\newlist{enumroman}{enumerate}{1}
\setlist*[enumroman]{label={\upshape(\roman*)}, nosep}

\DeclareMathOperator{\Ker}{Ker}
\DeclareMathOperator{\Cok}{Cok}
\renewcommand{\Im}{\operatorname{Im}}

\newcommand{\ol}{\overline}

\newcommand{\wh}{\widehat}

\newcommand{\xr}{\xrightarrow}
\newcommand{\N}{\mathbb N}
\newcommand{\Q}{\mathbb Q}
\newcommand{\F}{\mathbb F}

\begin{document}
\title[An application of Fontaine's monoidal maps to perfectoid towers]
{An application of Fontaine's monoidal maps to perfectoid towers}

\author[K. Hayashi]{Kazuki Hayashi}
\address{Department of Mathematics, Institute of Science Tokyo, 2-12-1 Ookayama, Meguro, Tokyo 152-8551, Japan}
\email{hazuki0694@gmail.com}

\author[S. Ishiro]{Shinnosuke Ishiro}
\address{National Institute of Technology, Gunma College, 580 Toriba-machi, Maebashi-shi, Gunma 371-8530, Japan}
\email{shinnosukeishiro@gmail.com}

\author[K. Shimomoto]{Kazuma Shimomoto}
\address{Department of Mathematics, Institute of Science Tokyo, 2-12-1 Ookayama, Meguro, Tokyo 152-8551, Japan}
\email{shimomotokazuma@gmail.com}

%\thanks{2000 {\em Mathematics Subject Classification\/}:}

%\keywords{Math}
%%%% 謝辞、キーワード、MSCコード  
%\thanks{本研究は科研費(課題番号:99999999)の助成を受けたものである。}

%\subjclass{13}
%\subjclass[2000]{Primary 13-XX}
%\subjclass[2000]{Primary ; Secondary}
%\date{\today \, (\printtime)}
%\date{\today}

\begin{abstract} 
%In his foundational work of $p$-adic Hodge theory, J.-M.\ Fontaine introduced an important map, called Fontaine's monoidal map to connect arithmetic objects in mixed characteristic with arithmetic objects in positive characteristic.
To connect arithmetic and ring-theoretic properties of rings of mixed characteristic with those of positive characteristic, we introduce monoidal maps for perfectoid towers.
Using these maps, we discuss the almost integrality of perfectoid towers and of their tilts.
We also show that the towers constructed by F. Andreatta via ramification theory become perfectoid towers, and we apply the monoidal maps to deduce the normality of their small tilts.
\end{abstract}

\maketitle
\setcounter{tocdepth}{2}
\tableofcontents

\section{Introduction}

We start the introduction by recalling Fontaine's monoidal map, which has been widely used in $p$-adic Hodge theory and the foundations of perfectoid spaces. Let $p>0$ be a prime and assume that $A$ is a $p$-adically complete ring. Then there is a well-defined multiplicative map $\sharp:A^\flat \to A$, where
$$
A^{\flat}:=\varprojlim_{i \ge 0}\big\{\cdots \to A/pA\xrightarrow{\Frob} A/pA \xrightarrow{\Frob} \cdots \xrightarrow{\Frob} A/pA\big\},
$$
which is the \textit{tilt} of $A$. 
In \cite{EHS24}, the authors showed that some ring-theoretic properties can be transferred from $A$ to $A^\flat$. 
As there is no direct way of constructing a ring map between $A$ and $A^\flat$, it is essential to exploit the multiplicativity of $\sharp:A^\flat \to A$ to study the above questions. The drawback is that if $A$ is assumed to be a Noetherian ring, then $A^\flat$ is too small to be useful due to the fact that the Frobenius map on $A/pA$ is far from being surjective. Thus, the strategy used in \cite{EHS24} works effectively mostly for big rings such as perfectoid rings.

On the other hand, the second and third authors \cite{INS25} with Kei Nakazato introduced the notion of \emph{perfectoid towers}, which gives a kind of Noetherization of perfectoid theory. A perfectoid tower $\{R_i,t_i\}_{i\geq 0}$ is a sequence of ring homomorphism $R_0\xr{t_0}R_1\xr{t_1}\cdots$ and a principal ideal $I_0=(f_0)\subset R_0$ satisfying seven axioms (\cref{def:perfectoid tower}). Such a tower approximates a perfectoid ring: the $I_0$-adically completion $\wh{R_\infty}$ of the colimit $R_\infty=\varinjlim_{i\geq 0}R_i$ is a perfectoid ring (\cref{ntn:tower}). In particular, we can consider Fontaine's monoidal map $\sharp\colon (\wh{R_\infty})^\flat\to \wh{R_\infty}$.
Our aim of this paper is to exploit this map for perfectoid towers. 
Due to \cite{INS25}, there is the notion of the tilt $\{R_i^{s.\flat},t_i^{s.\flat}\}_{i\geq 0}$ of a perfectoid tower $\{R_i,t_i\}_{i\geq 0}$, and each layer $R_i^{s.\flat}$, called the $i$-th \emph{small tilt}, can be embedded in $(\wh{R_\infty})^\flat$. In this context, we construct a refined version $\sharp^{(i)}\colon R_i^{s.\flat}\to R_i+I_0\wh{R_\infty}$ of the map $\sharp\colon (\wh{R_\infty})^\flat\to \wh{R_\infty}$ (\cref{def:refined monoidal}). Although $\sharp^{(i)}$ is obtained by restricting $\sharp$, it is compatible with some known results (\cref{prop:sharpi2}).
Moreover, we apply this notion to show that the property being completely integrally closed (\cref{def:integral}) can be transferred from $\{R_i,t_i\}_{i\geq 0}$ to $\{R_i^{s.\flat},t_i^{s.\flat}\}_{i\geq 0}$.

%The main result of \cite{EHS24} (see \cref{thm:sharp} (3)) shows that if $\wh{R_\infty}$ is completely integrally closed in $\wh{R_\infty}[\frac{1}{f_0}]$, then $(\wh{R_\infty})^\flat$ is completely integrally closed in $(\wh{R_\infty})^\flat[\frac{1}{f_0^{s.\flat}}]$, where $f_0^{s.\flat}$ is an element corresponding to $f_0$ via $\sharp$.
%In this paper, we introduce a refined version of this map to deal with singularities of Noetherian rings. Namely, we construct a map analogous to $\sharp:A^\flat \to A$ for an \textit{asymptotic perfecrtoid tower}, which is defined in \cite{IS25}. Then we use this construction to prove some basic results. Note that the main results in this article have applications to lim Cohen-Macaulay sequences (see \cite{IS25} for details).

\begin{theorem}[{\cref{lem:citower,tiltingintegral}}]
\label{thm:A}
Let $\{R_i,t_i\}_{i\geq 0}$ be a perfectoid tower arising from some pair $(R,f_0)$ with $f_0\in R$. If $R$ is $f_0$-torsion free, then the following conditions are equivalent.
  \begin{enumerate}
  \item $R_i$ is completely integrally closed in $R_i[\frac{1}{f_0}]$ for every $i\geq 0$.
  \item $R_\infty$ is completely integrally closed in $R_\infty[\frac{1}{f_0}]$.
  \item $\wh{R_\infty}$ is completely integrally closed in $\wh{R_\infty}[\frac{1}{f_0}]$.
  \end{enumerate}
Moreover, if $\{R_i,t_i\}_{i\geq 0}$ satisfies these conditions, then so does the tilt $\{R_i^{s.\flat},t_i^{s.\flat}\}_{i\geq 0}$.
\end{theorem}

Since the notion of complete integral closedness coincides with that of integral closedness in the Noetherian case, the reader might expect that \cref{thm:A} can be applied to conclude the normality of small tilts $R_i^{s.\flat}$. As a second main theorem, we give such an example based on ramification theory. More precisely, we construct the following type of perfectoid towers and establish the normality of its small tilts.

\begin{theorem}[{\cref{thm:Sn perfectoid,cor:B}}]
\label{thm:B}
Let $R$ be an unramified complete regular local ring of mixed characteristic $(0,p)$ whose residue field $k$ is perfect. Let $R\hookrightarrow S$ be a finite extension of normal local domains such that $R[\frac{1}{p}]\hookrightarrow S[\frac{1}{p}]$ is \'{e}tale. Let $\{R_n\}_{n\geq 0}$ be the perfectoid tower associated to $R$, and let $S_n$ be the integral closure of $R_n$ in $(R_n\otimes_RS)[\frac{1}{p}]$ for every $n\geq 0$. Then there exist a rational $\varepsilon \in (0,1)\cap \Q$ and an integer $N\geq 0$ such that $\{S_n\}_{n\geq N}$ is a perfectoid tower arising from $(S_N,(p^\varepsilon))$. Moreover, the small tilt $S_n^{s.\flat}$ is normal for any $n\geq N$.
\end{theorem}

Note that the construction and the proof of \cref{thm:B} are heavily based on the work of Andreatta \cite{An06}. Essentially, he had already constructed certain ramified towers as in $\{S_n\}_{n\geq 0}$ appearing in \cref{thm:B} and utilize it to develop generalized theory of field of norms. However, our work connects his study to the theory of perfectoid towers; moreover, our proof of the normality of small tilts $S_n^{s.\flat}$ is based on \cref{thm:A} and two fundamental theorems in perfectoid theory, (Witt-perfect) almost purity theorem and perfectoid-\'{e}tale correspondence.

%In \Cref{s:A}, we record a proof of some results on the flatness of completions of not necessarily Noetherian rings, which are needed in \Cref{s:4}.
\ \\
%\subsection*{Notations and Conventions}
\textbf{Notations and Conventions}
Throughout this paper, we follow the notation and the convention stated below.
  \begin{itemize}
  \item We consistently fix a prime number $p$.
  \item All rings are assumed to be commutative and unital (unless otherwise stated).
  \item A \emph{pair} is a couple $(A,I)$ consisting of a ring $A$ and an ideal $I$ of $A$. When the ideal $I$ is principal, say $I=(a)$, then we often write $(A,a)$ in place of $(A,I)$.
  \item For a pair $(A,I)$ and an $A$-module $M$, we say that an element $x\in M$ is \emph{$I$-torsion} if for all $a\in I$ there exists an integer $n>0$ such that $a^nx=0$. Let $M_{\textrm{$I$-}\mathrm{tor}}$ denote the $A$-submodule of $M$ consisting of all $I$-torsion elements in $M$. We say that $M$ is \emph{$I$-torsion free} if $M_{\textrm{$I$-}\mathrm{tor}}=0$.
  \item For a pair $(A,I)$, when we say an $A$-module $M$ is $I$-adically complete, we always mean that $M$ is Hausdorff complete with respect to the $I$-adic topology.
  \end{itemize}

\begin{acknowledgement}
The authors are deeply grateful to Kei Nakazato for his valuable comments, particularly his insights regarding the construction of monoidal maps on perfectoid towers.
The authors would also like to thank Ryo Ishizuka for his advice.
\end{acknowledgement}
%%%%%%%%%%%%%%%%%%%%%%%%%%%%%%%%%%%%%%%%%%%%%%%%%%%%%%%%%%%%%%%%%%%%%%%
%%%%%%%%%%%%%%%%%%%%%%%%%%%%%%%%%%%%%%%%%%%%%%%%%%%%%%%%%%%%%%%%%%%%%%%

%「ネーター性や正則性がtiltで遺伝することは\cite{INS25}でわかっている」ということにもどこかで触れたい。

%%%%%%%%%%%%%%%%%%%%%%%%%%%%%%%%%%%%%%%%%%%%%%%%%%%%%%%%%%%%%%%%%%%%%%%
%%%%%%%%%%%%%%%%%%%%%%%%%%%%%%%%%%%%%%%%%%%%%%%%%%%%%%%%%%%%%%%%%%%%%%%
\section{Almost integrality in cartesian diagrams}
\label{s:2}
In this section, we discuss behavior of almost integrality in cartesian diagrams of rings. The results in this section are essential of the proof of \cref{thm:A}.

\begin{definition}
\label{def:integral}
Let $A\subseteq B$ be a ring extension.
  \begin{enumerate}
  \item Fix a positive integer $n>0$. We say that $A$ is \emph{$n$-root closed in $B$} if every $b\in B$ with $b^n\in A$ lies in $A$.
  \item An element $b\in B$ is \emph{integral over $A$} if the $A$-subalgebra $A[b]$ of $B$ is a finitely generated $A$-module. We say that $A$ is \emph{integrally closed in $B$} if every $b\in B$ that is integral over $A$ lies in $A$.
  \item An element $b\in B$ is \emph{almost integral over $A$} if the $A$-subalgebra $A[b]$ of $B$ is contained in a finitely generated $A$-module. We say that $A$ is \emph{completely integrally closed in $B$} if every $b\in B$ that is almost integral over $A$ lies in $A$.
  \end{enumerate}
\end{definition}

By definition, if $A$ is completely integrally closed in $B$, then it is integrally closed and $n$-root closed in $B$. If $A$ is Noetherian, then it is completely integrally closed in $B$ if and only if it is integrally closed in $B$.

The properties above are preserved by pullback in the following sense.

\begin{lemma}
\label{lem:ci cartesian}
Consider a cartesian diagram of rings
\[
\begin{tikzcd}
A \rar["\varphi"] \dar[hookrightarrow] \ar[rd,phantom,"\mathrm{p.b.}"] & B \dar[hookrightarrow] \\
A' \rar["\varphi'"'] & B',
\end{tikzcd}
\]
where the vertical maps are injective. If $B$ is $n$-root closed (resp.\ integrally closed, completely integrally closed) in $B'$, then $A$ is $n$-root closed (resp.\ integrally closed, completely integrally closed) in $A'$.
\end{lemma}

\begin{proof}
Assume that $B$ is $n$-root closed in $B'$. Suppose $a'\in A'$ satisfies $a'^n \in A$. Then $\varphi(a'^n)\in B$. But $\varphi(a'^n) = \varphi'(a'^n) = \varphi'(a')^n$, and so $\varphi'(a')\in B$. Hence $a'\in \varphi'^{-1}(B)=A$. Thus $A$ is $n$-root closed in $A'$.

Next assume $B$ is integrally closed (resp.\ completely integrally closed) in $B'$. If $a'\in A'$ is integral (resp.\ almost integral) over $A$, then $\varphi'(a')\in B'$ is integral (resp.\ almost integral) over $B$, and so $\varphi'(a')\in B$. Hence $a'\in \varphi'^{-1}(B)=A$. Thus $A$ is integrally closed (resp.\ completely integrally closed) in $A'$.
\end{proof}

We need the following criterion of cartesian diagrams in the torsion free case.

\begin{lemma}
\label{lem:cartesian}
Let $R$ be a ring and $f\in R$. Let $\varphi\colon M\to N$ be a homomorphism of $f$-torsion free $R$-modules. Then the following conditions are equivalent.
  \begin{enumerate}
  \item The induced homomorphism $M/fM\to N/fN$ is injective.
  \item The commutative diagram of $R$-modules
  \[
  \begin{tikzcd}
  M \rar["\varphi"] \dar[hookrightarrow] & N \dar[hookrightarrow] \\
  M[\frac{1}{f}] \rar["\varphi_f"'] & N[\frac{1}{f}]
  \end{tikzcd}
  \]
  is cartesian.
  \end{enumerate}
\end{lemma}

\begin{proof}
Since both $M$ and $N$ are $f$-torsion free, we have the commutative diagram with exact rows
\[
\begin{tikzcd}
0 \rar & M \rar["f"] \dar["\varphi"'] & M \rar \dar["\varphi"] & M/fM \rar \dar & 0 \\
0 \rar & N \rar["f"'] & N \rar & N/fN \rar & 0.
\end{tikzcd}
\]
Observe that the condition (1) is equivalent to the condition that the left-hand square is cartesian. Hence the implication ``(1) $\Rightarrow$ (2)'' follows from the fact that filtered colimits commute with finite limits in the category of $R$-modules. To show the converse, consider the commutative diagram
\[
\begin{tikzcd}
M \rar["f"] \dar["\varphi"'] & M \rar[hookrightarrow] \dar["\varphi"] & M[\frac{1}{f}] \dar["\varphi_f"] \\
N \rar["f"] & N \rar[hookrightarrow] & N[\frac{1}{f}].
\end{tikzcd}
\]
If the condition (2) holds, then the right-hand square is cartesian. Moreover, since the multiplication by $f$ on $M[\frac{1}{f}]$ is an isomorphism, we see that the outer square is also cartesian. Hence the left-hand square is cartesian, that is, the condition (1) holds.
\end{proof}

The similar assertion holds for algebras:

\begin{proposition}
\label{prop:cartesian2}
Let $\varphi\colon A\to B$ be a ring homomorphism such that $A$ and $B$ are $f$-torsion free for an element $f\in A$. Then the following conditions are equivalent.
  \begin{enumerate}
  \item The induced homomorphism $A/fA\to B/fB$ is injective.
  \item The commutative diagram of of rings
  \begin{equation}
  \label{eq:pb}
  \begin{tikzcd}
  A \rar["\varphi"] \dar[hookrightarrow] & B \dar[hookrightarrow] \\
  A[\frac{1}{f}] \rar["\varphi_f"'] & B[\frac{1}{f}]
  \end{tikzcd}
  \end{equation}
  is cartesian.
  \end{enumerate}
\end{proposition}

\begin{proof}
Since the forgetful functor from the category of rings to the category of abelian groups preserves and reflects finite limits and filtered colimits, the proposition can be reduced to \cref{lem:cartesian}.
\end{proof}

The following lemma will be useful later.

\begin{lemma}
\label{lem:ci cartesian2}
Let $A\to B$ be a ring homomorphism such that $A$ and $B$ are $f$-torsion free for an element $f\in A$. Assume that the diagram of rings
\[
\begin{tikzcd}
A \rar["\varphi"] \dar[hookrightarrow] & B \dar[hookrightarrow] \\
A[\frac{1}{f}] \rar["\varphi_f"'] & B[\frac{1}{f}]
\end{tikzcd}
\]
is cartesian (i.e., $\varphi_f^{-1}(B)=A$). If $A$ is completely integrally closed in $A[\frac{1}{f}]$, then every element of $\Im(\varphi_f)$ that is almost integral over $B$ lies in $\Im\varphi$.
\end{lemma}

\begin{proof}
Let $y=\varphi_f(x)\in\Im(\varphi_f)$ be almost integral over $B$, where $x\in A[\frac{1}{f}]$. There exists some $c>0$ such that $f^cx\in A$ and $f^cy^n\in B$ for all $n>0$. Then $\varphi_f(f^cx^n) = f^cy^n\in B$, and thus $f^cx^n\in \varphi_f^{-1}(B)=A$ for all $n>0$. Consequently, $x\in A[\frac{1}{f}]$ is almost integral over $A$, and therefore $x\in A$ by assumption. Hence, $y=\varphi_f(x)=\varphi(x)\in \Im\varphi$, as claimed.
\end{proof}

For the cartesian diagram of type \eqref{eq:pb}, the converse of \cref{lem:ci cartesian} holds: 

\begin{proposition}
\label{prop:ciBL}
Let $A$ be a ring, and $f\in A$ a non-zero-divisor. Let $\varphi\colon A\to B$ be a $f$-torsion free $A$-algebra that induces an isomorphism $A/fA\xr{\cong}B/fB$. Then the following conditions are equivalent.
  \begin{enumerate}
  \item $A$ is completely integrally closed in $A[\frac{1}{f}]$.
  \item $B$ is completely integrally closed in $B[\frac{1}{f}]$.
  \end{enumerate}
\end{proposition}

\begin{proof}
Note that we are in the situation as in \cref{lem:ci cartesian2} by \cref{prop:cartesian2}.

(2) $\Rightarrow$ (1): This implication follows from \cref{lem:ci cartesian}.

(1) $\Rightarrow$ (2): This implication has been essentially established in the proof of \cite[Corollary 2.7]{NS18}; we provide the full proof for the reader's convenience.
Pick an element $x\in B[\frac{1}{f}]$ that is almost integral over $B$. Then we have $f^dx\in B$ for some $d>0$. Since $B=\varphi(A)+f^dB$ by assumption, there exists some $a\in A$ such that $f^dx - \varphi(a)\in f^dB$ or equivalently, $x-\varphi_f\left(\frac{a}{f^d}\right)\in B$. Then, since $x$ and any element in $B$ are almost integral over $B$, we find that $\varphi_f\left(\frac{a}{f^d}\right)$ is almost integral over $B$. But then $\varphi_f\left(\frac{a}{f^d}\right)\in B$ by \cref{lem:ci cartesian2}. Hence $x\in B$, as desired.
\end{proof}

As a special case, we obtain the lemma of Beauville--Laszlo (cf.\ \cite[Lemma 2.6]{NS18}) and \cite[Corollary 2.8]{NS18}.

\begin{corollary}
\label{cor:BLlemma}
Let $A$ be a ring, and $f\in A$ a non-zero-divisor. Let $\varphi\colon A\to \wh{A}$ denote the $f$-adic completion of $A$. Then the following assertions hold.
  \begin{enumerate}
  \item $\wh{A}$ is $f$-torsion free.
  \item The commutative diagram of rings
  \[
  \begin{tikzcd}
  A \rar["\varphi"] \dar[hookrightarrow] & \wh{A} \dar[hookrightarrow] \\
  A[\frac{1}{f}] \rar["\varphi_f"'] & \wh{A}[\frac{1}{f}]
  \end{tikzcd}
  \]
  is cartesian.
  \item $A$ is completely integrally closed in $A[\frac{1}{f}]$ if and only if $\wh{A}$ is completely integrally closed in $\wh{A}[\frac{1}{f}]$.
  \end{enumerate}
\end{corollary}

\begin{proof}
(1) is easy; see \cite[Chapter II, Lemma 1.1.5]{FKI}. The assertions (2) and (3) follow from \Cref{prop:cartesian2,prop:ciBL}, respectively.
\end{proof}

Finally, we have the following statement.

\begin{proposition}
\label{prop:ci cartesian3}
Let $\{A_i\}_{i\in I}$ be an inductive system of rings indexed by a directed set $I$, and set $A\coloneqq\varinjlim_{i\in I}A_i$. Fix $i\in I$, and choose an element $f\in A_i$. Suppose that $A_j$ is $f$-torsion free for any $j\geq i$ (hence $A$ is $f$-torsion free). If the diagrams of rings
\[
\begin{tikzcd}
A_j \rar \dar & A \dar \\
A_j[\frac{1}{f}] \rar & A[\frac{1}{f}]
\end{tikzcd}
\]
is cartesian for any $j\geq 0$, then the following conditions are equivalent.
  \begin{enumerate}
  \item $A_j$ is completely integrally closed in $A_j[\frac{1}{f}]$ for any $j\geq i$.
  \item $A$ is completely integrally closed in $A[\frac{1}{f}]$.
  \item The $f$-adically completion $\wh{A}$ is completely integrally closed in $\wh{A}[\frac{1}{f}]$.
  \end{enumerate}
\end{proposition}

\begin{proof}
(2) $\Rightarrow$ (1): This follows from \cref{lem:ci cartesian}.

(1) $\Rightarrow$ (2): Since every element in $A[\frac{1}{f}]$ is represented by an element $A_j[\frac{1}{f}]$ for some $j\geq i$, the assertion follows from \cref{lem:ci cartesian2}.

(2) $\Leftrightarrow$ (3): This is \cref{cor:BLlemma}(3).
\end{proof}

%%%%%%%%%%%%%%%%%%%%%%%%%%%%%%%%%%%%%%%%%%%%%%%%%%%%%%%%%%%
%%%%%%%%%%%%%%%%%%%%%%%%%%%%%%%%%%%%%%%%%%%%%%%%%%%%%%%%%%%
\section{Monoidal maps for perfectoid towers}
\label{s:3}

In this section, we briefly review Fontaine's monoidal maps and perfectoid towers, construct a refined version of monoidal map, and then give a proof of \cref{thm:A}.

%%%%%%%%%%%%%%%%%%%%%%%%%%%%%%%%%%%%%%%%%%%%%%%%%%%%%%%%%%%
\subsection{Review on Fontaine's monoidal maps}

We say that an $\F_p$-algebra is \emph{perfect} if its absolute Frobenius is bijective. One can easily show (cf.\ \cite[Lemma 2.2]{EHS24}) that every perfect $\F_p$-algebra is reduced. Given any ring, we have a perfect $\F_p$-algebra in the following way.

\begin{definition}
Let $A$ be a ring. Then we denote by $A^\flat$ the $\F_p$-algebra
\[
A^\flat \coloneqq \varprojlim_{\mathrm{Frob}}A/pA = \varprojlim(\cdots \xr{\mathrm{Frob}} A/pA \xr{\mathrm{Frob}} A/pA)
\]
This is called the \emph{inverse perfection} of $A$ (when $A$ is an $\F_p$-algebra) or the \emph{tilt} of $A$ (when $A$ is not an $\F_p$-algebra).
\end{definition}

One can easily show (cf.\ \cite[Proposition 2.6.1]{EHS24}) that $A^\flat$ is a perfect $\F_p$-algebra.
If $A$ is $p$-adically complete, then $A^\flat$ admits another description as follows.

\begin{lemma}
\label{lem:monoidal}
Let $A$ be an $I$-adically complete ring for an ideal $I\subseteq A$ containing $p$ (note that it follows that $A$ is $p$-adically complete by \cite[Chapter 0, Proposition 7.2.5]{FKI}). Then, in the commutative diagram of topological monoids induced by the canonical projections
\[
\begin{tikzcd}
\varprojlim\limits_{x\mapsto x^p}A \rar \ar[rd] & A^\flat \dar \\
 & \varprojlim\limits_{\mathrm{Frob}}A/I,
\end{tikzcd}
\]
all arrows are isomorphisms, where $A/pA$ and $A/I$ are considered with the discrete topology.
\end{lemma}

\begin{proof}
Since $A$ is $p$-adically complete, it suffices to show that $\varprojlim\limits_{x\mapsto x^p}A \to \varprojlim\limits_{\mathrm{Frob}}A/I$ is an isomorphism. To show this, one can check that, as in \cite[Lemma 2.5]{EHS24}, the map
\begin{equation}
\label{eq:inverse}
(x_n\ \mathrm{mod}\ I)_{n\geq 0} \mapsto \lim_{n\to\infty}x_n^{p^n}
\end{equation}
gives the inverse mapping; see also \cite[Lemma 2.4]{ALB20}.
\end{proof}

By the above lemma, we have the following notion.

\begin{definition}
Let $A$ be a $p$-adically complete ring. We define the \emph{monoidal map} $\sharp\colon A^\flat \to A$ as  the composite map of topological monoids
\[
\sharp \colon A^\flat \xleftarrow{\cong} \varprojlim_{x\mapsto x^p}A \xr{\pr_0} A,
\]
where the first map is the isomorphism from \cref{lem:monoidal} and the second map is the $0$-th projection.
\end{definition}

\begin{proposition}
Let $A$ be a $p$-adically complete ring. Then $\sharp\colon A^\flat\to A$ induces a bijection on the sets of idempotents $\mathrm{Idem}(A^\flat) \xr{\sim} \mathrm{Idem}(A)$. In particular, $A$ is connected if and only if so is $A^\flat$.\footnote{It can be easily shown (cf.\ \cite[Proposition 2.6.2]{EHS24}) that if $A$ is an integral domain, then so is $A^\flat$.}%footnote{Note that $A^\flat\neq 0$ if $A$ is $p$-adically complete and $A\neq 0$.}}.
\end{proposition}

\begin{proof}
The map $A \to A^{\flat} : e\mapsto(e,e,e,\ldots)$ gives the inverse of the map $\sharp$.
\end{proof}

Let us list some fundamental results on the monoidal map.

\begin{theorem}
\label{thm:sharp}
Let $A$ be a perfectoid ring that is $\varpi$-adically complete for a $\varpi\in A$ with $p\in\varpi^pA$. Let $\varpi^\flat\in A^\flat$ such that $(\varpi^\flat)^\sharp$ is a unit multiple of $\varpi$.
  \begin{enumerate}
  \item $\sharp\colon A^\flat\to A$ induces an isomorphism of $\F_p$-algebras
  \[
  A^\flat/(\varpi^\flat)^p A^\flat \xr{\cong} A/\varpi^p A.
  \]
  In particular, the ideal $(\varpi^\flat)$ is independent of the choice of $\varpi^\flat$.
  \item $\sharp\colon A^\flat\to A$ induces an isomorphism of abelian groups $(A^\flat)_{\textrm{$\varpi^\flat$-}\mathrm{tor}} \xr{\cong} A_{\textrm{$\varpi$-}\mathrm{tor}}$. In particular, $A$ is $\varpi$-torsion free if and only if $A^\flat$ is $\varpi^\flat$-torsion free.
  \item If $A$ is $\varpi$-torsion free and $A$ is completely integrally closed (resp.\ integrally closed) in $A[\frac{1}{\varpi}]$, then $A^\flat$ is completely integrally closed (resp.\ integrally closed) in $A^\flat[\frac{1}{\varpi^\flat}]$.
  \end{enumerate}
\end{theorem}

\begin{proof}
The assertion (1) is \cite[(2.1.2.5)]{CS24} (essentially, the proof of \cite[Lemma 3.10 (i)]{BMS1}). 
The assertion (2) is \cite[(2.1.2.8)]{CS24}.
The assertion (3) is \cite[Main theorem 1.3 and 1.4]{EHS24}.
\end{proof}

%%%%%%%%%%%%%%%%%%%%%%%%%%%%%%%%%%%%%%%%%%%%%%%%%%%%%%%%%%%
\subsection{Perfectoid towers}

By a \emph{tower of rings} we mean an inductive system of rings $R_0\xr{t_0}R_1\xr{t_1}R_2\xr{t_2}\cdots$ indexed by $\N$, which we denote by $\{R_i,t_i\}_{i\geq 0}$ or simply $\{R_i\}_{i\geq 0}$.

\begin{definition}[{\cite[Definition 3.4, Definition 3.21]{INS25}}]
\label{def:perfectoid tower}
%Fix a prime number $p$. 
Let $R$ be a ring, and $I\subset R$ an ideal. A tower of rings $\{R_i,t_i\}_{i \ge 0}$ is called a \emph{purely inseparable tower arising from $(R,I)$} if it satisfies the following three axioms.
  \begin{enumalph}
  \item $R_0=R$ and $p\in I$.
  \item For any $i\geq 0$, the ring map $\ol{t_i} \colon R_i/IR_i \to R_{i+1}/IR_{i+1}$ induced by $t_i$ is injective.
  \item For any $i\geq 0$, the Frobenius endomorphism $F\colon R_{i+1}/IR_{i+1}\to R_{i+1}/IR_{i+1}$ factors as
  \[
  \begin{tikzcd}
  R_{i+1}/IR_{i+1} \rar["F"] \ar[rd,"F_i"',dashed] & R_{i+1}/IR_{i+1} \\
   & R_i/IR_i \uar["\ol{t_i}"']
  \end{tikzcd}
  \]
  We call the map $F_i$ (which is unique by the axiom (b)) \emph{the $i$-th Frobenius projection}.
  \end{enumalph}
The tower of rings $\{R_i,t_i\}_{i\geq 0}$ is called a \emph{perfectoid tower arising from $(R,I)$} if it satisfies in addition the following axioms.
  \begin{enumalph}[resume]
  \item For any $i\geq 0$, the $i$-th Frobenius projection $F_i\colon R_{i+1}/IR_{i+1}\to R_i/IR_i$ is surjective.
  \item For any $i\geq 0$, $R_i$ is $IR_i$-adically Zariskian (i.e., $IR_i$ is contained in the Jacobson radical of $R_i$).
  \item $I$ is a principal ideal, and there exists a principal ideal $I_1\subset R_1$ satisfying the following conditions.
    \begin{enumerate}
    \item[(f-1)] $I_1^p=IR_1$.
    \item[(f-2)] For any $i\geq 0$, $\Ker(F_i)=I_1(R_{i+1}/IR_{i+1})$.
    \end{enumerate}
  \item For any $i\geq 0$, $I(R_i)_{I\textrm{-}\mathrm{tor}}=0$, and there exists a (unique) bijection $(F_i)_{\mathrm{tor}}\colon (R_{i+1})_{I\textrm{-}\mathrm{tor}} \to (R_i)_{I\textrm{-}\mathrm{tor}}$ such that the following diagram commutes:
  \[
  \begin{tikzcd}
  (R_{i+1})_{I\textrm{-}\mathrm{tor}} \rar \dar["(F_i)_{\mathrm{tor}}"',dashed] & R_{i+1}/IR_{i+1} \dar["F_i"] \\
  (R_i)_{I\textrm{-}\mathrm{tor}} \rar & R_i/IR_i
  \end{tikzcd}
  \]
  \end{enumalph}
\end{definition}

Given a purely inseparable tower, we have its \emph{quasi-inverse perfection}.

\begin{definition}[{\cite[Definition 3.8]{INS25}}]
Let $\{R_i,t_i\}_{i\geq 0}$ be a purely inseparable tower arising from some pair $(R,I)$. 
\begin{enumerate}
    \item For any $j\geq 0$, we define the \textit{$j$-th inverse quasi-perfection of $\{R_i,t_i\}_{i\geq 0}$ associated to $(R,I)$} as the limit of topological rings
    \[
    R_j^{q.\textnormal{frep}} \coloneqq \varprojlim\big\{\cdots \to R_{j+i+1}/I R_{j+i+1}\xrightarrow{F_{j+i}} R_{j+i}/IR_{j+i}
    \xrightarrow{F_{i-1}} \cdots \xrightarrow{F_j} R_j/IR_j\big\},
    \]
where we endow each $R_i/IR_i$ with the discrete topology.
    \item For any $j \geq 0$, we define an injective ring map $t_j^{q.\textnormal{frep}} : R_j^{q.\textnormal{frep}} \hookrightarrow R_{j+1}^{q.\textnormal{frep}}$ by the rule
    \[
    t_j^{q.\textnormal{frep}} ((a_i)_{i \geq 0} ) \coloneqq (\overline{t}_{j+i}(a_i))_{i \geq 0}.
    \]
    \item We call the tower $\{ R_i^{q.\textnormal{frep}}, t_i^{q.\textnormal{frep}} \}_{i \geq 0}$ \textit{the inverse perfection of $\{R_i, t_i \}_{i \geq 0}$ associated to $(R,I)$}.
\end{enumerate}

If $\{ R_i, t_i \}_{i \geq 0}$ is a perfectoid tower, then the $j$-th inverse quasi-perfection is called \textit{the $j$-th small tilt} and denoted by $R_j^{s.\flat}$.
The ring map $t_j^{q.\textnormal{frep}}$ is also denoted by $t_j^{s.\flat}$.
Then the tower $\{ R_i^{s.\flat}, t_i^{s.\flat} \}_{i \geq 0}$ is called \textit{the tilt of $\{ R_i, t_i \}_{i \geq 0}$ associated to $(R,I)$}.
We also define the ideal $I^{s.\flat}$ of $R_0^{s.\flat}$ as the kernel of the $0$-th projection $R_0^{s.\flat}\to R/I$.
%:= \ker(\pi_i \circ \Phi^{(i)}_0 )$ of $R_i^{s.\flat}$ where $\pi_i : R_i/I_0R_i \to R_i/I_iR_i$ is the natural projection.
\end{definition}

We note that the tilt $\{R_i^{s.\flat},t_i^{s.\flat} \}_{i \geq 0}$ of a perfectoid tower $\{R_i, t_i\}_{i \geq 0}$ arising from some pair $(R,I)$ is also a perfectoid tower arising from $(R_0^{s.\flat}, I_0^{s.\flat})$ (see \cite[Corollary 3.52]{INS25}).
We give some examples.

\begin{example}\label{eg:RLRTower}
  \begin{enumerate}
  \item For a reduced $\F_p$-algebra $R$, the tower $R\xr{\mathrm{Frob}}R\xr{\mathrm{Frob}}R\xr{\mathrm{Frob}}\cdots$ consisting of the Frobenius endomorphisms is a perfectoid tower arising from $(R,(0))$. Conversely, a perfectoid tower arising from some pair $(R,(0))$ is of the form above up to isomorphism (\cite[Lemma 3.24]{INS25}).
  \item The following example is one of the most typical examples in mixed characteristic and used in the next section.
  Let $R$ be a $d$-dimensional unramified complete regular local ring with perfect residue field $k$ of characteristic $p$. Then we have $R\cong W(k)[[x_2,\ldots,x_d]]$. Consider the tower
  \[
  R_0\coloneqq R \hookrightarrow R_1\coloneqq R[p^{1/p},x_2^{1/p},\ldots,x_d^{1/p}]\hookrightarrow\cdots\hookrightarrow R_i\coloneqq R[p^{1/p^i},x_2^{1/p^i},\ldots,x_d^{1/p^i}]\hookrightarrow\cdots,
  \]
  where each transition map is the canonical injection. Then $R_i/pR_i \cong k[[x_2,\ldots,x_d]][x_2^{1/p^i},\ldots,x_d^{1/p^i}]$ for every $i\geq 0$. We deduce from this that the tower $\{R_i\}_{i\geq 0}$ is a perfectoid tower arising from $(R,(p))$. In this case, the Frobenius projection $F_i\colon R_{i+1}/pR_{i+1}\to R_i/pR_i$ is given as the $p$-th power map. Its tilt $R_0^{s.\flat}\to R_1^{s.\flat}\to R_2^{s.\flat}\to\cdots$ is given by
  \[
  k[[p^\flat,x_2^\flat,\ldots,x_d^\flat]] \hookrightarrow k[[(p^\flat)^{1/p},(x_2^\flat)^{1/p},\ldots,(x_d^\flat)^{1/p}]] \hookrightarrow k[[(p^\flat)^{1/p^2},(x_2^\flat)^{1/p^2},\ldots,(x_d^\flat)^{1/p^2}]] \hookrightarrow \cdots 
  \]
  where $p^\flat,x_2^\flat,\ldots,x_d^\flat$ are variables corresponding to $p,x_2,\ldots,x_d$. By a similar construction, we obtain perfectoid towers arising from local log-regular rings (\cite[\S3F]{INS25}).
  \item Ishizuka \cite{Ish24} gives a construction of perfectoid towers generated by Frobenius lifts. Due to this fact, one can obtain perfectoid towers arising from non-normal or non-Cohen--Macaulay rings. For other examples including ones that have $p$-torsion elements, see \cite[\S5]{IS25}.
  \end{enumerate}
\end{example}

The following three lemmas give useful sufficient conditions to verify axioms (b), (e), and (f) in \cref{def:perfectoid tower}.

\begin{lemma}[{cf.\ the proof of \cite[Lemma 4.11]{NS18}}]
\label{lem:injmod}
Let $R\hookrightarrow S$ be an injective map of rings, and $f\in R$ a non-zero-divisor. If $R$ is integrally closed in $R[\frac{1}{f}]$ and $S$ is integral over $R$, then the induced map $R/fR\to S/fS$ is injective.
\end{lemma}

\begin{proof}
We have to show that $fS\cap R = fR$. Pick $r\in fS\cap R$. We can write $r=fs$ for some $s\in S$. Then $s=r/f\in R[\frac{1}{f}]$, which is integral over $R$. Hence $s\in R$, and so $r=fs\in fR$.
\end{proof}

For a ring $A$, let $\rad(A)$ denote its Jacobson radical.

\begin{lemma}
\label{lem:e}
Let $f\colon R\to S$ be an integral map of rings, and $I\subseteq R$ an ideal.
  \begin{enumerate}
  \item $\rad(R)\subset\rad(S)\cap R$ with equality if $f$ is injective.
  \item If $R$ is $I$-adically Zariskian, then $S$ is $IS$-adically Zariskian.
  \end{enumerate}
\end{lemma}

\begin{proof}
%For a ring $A$, let $\rad(A)$ denote its Jacobson radical. The extension $f(R)\subseteq S$ is integral, and thus $\rad (f(R)) = \rad(S)\cap f(R)$ (). But $\rad(f(R)) = f(\rad(R))$, and so we conclude that $\rad(R)\subseteq \rad(S)\cap R$, which implies the assertion.
(1) If $f$ is injective, then the assertion follows from \cite[Chapter V, \S2.1, Corollary 3]{BouAC}.
In the general case, we only have to notice that $f(\rad R) = \rad f(R)$, which follows from the surjectivity of $f$.
(2) is immediate from (1).
\end{proof}

\begin{lemma}
\label{lem:f-2}
Let $\{R_i,t_i\}_{i\geq 0}$ be a purely inseparable tower arising from some pair $(R,I_0)$. Assume that $I_0$ is a principal ideal, and $R_1$ contains a principal ideal $I_1=(f_1)$ that satisfies (f-1). Furthermore, assume that the image of $f_1$ in $R_i$ is a non-zero-divisor. Then $I_1$ satisfies (f-2) if and only if $R_{i+1}$ is $p$-root closed in $R_{i+1}[\frac{1}{f_1}]$ for every $i\geq 0$.
\end{lemma}

\begin{proof}
Since (f-2) is equivalent to say that the kernel of the Frobenius endomorphism on $R_{i+1}/I_0R_{i+1} = R_{i+1}/f_1^pR_{i+1}$ is generated by $f_1$, the assertion follows from the general fact: if $A$ is a ring and $t\in A$ is a non-zero-divisor with $p\in t^pA$, then $A$ is $p$-root closed in $A[\frac{1}{t}]$ if and only if the Frobenius endomorphism on $A/t^pA$ is generated by $t$ (\cite[(2.1.7.1)]{CS24}). 
\end{proof}

%Finally, let us include a few facts on perfectoid towers.

\begin{notation}
\label{ntn:tower}
Throughout this paper, we use the following notations for towers of rings.
  \begin{enumerate}
  \item Consider a tower of rings $\{R_i,t_i\}_{i\geq 0}$ satisfying axiom (a) for some pair $(R,I)$.
    \begin{itemize}
    \item We set $R_\infty \coloneqq\varinjlim \{ R_0 \xrightarrow{t_0}R_1\xrightarrow{t_1} R_2 \xrightarrow{t_2} \cdots\}$.
    %\item For any $i \geq 0$, we denote the canonical map $R_i \to R_\infty$ by $\vil_{j \geq i} t_{j}$.
    \item Let $\wh{R_\infty}$ denote the $I$-adically completion of $R_\infty$. This is a perfectoid ring if $\{R_i,t_i\}_{i\geq 0}$ is a perfectoid tower (\cite[Corollary 3.52]{INS25}).
    \end{itemize}
  \item Consider a purely inseparable tower $\{R_i,t_i\}_{i\geq 0}$ arising from some pair $(R,I)$.
    \begin{itemize}
    \item For any $j\geq 0$ and for any $m\geq 0$, we denote by $\Phi^{(j)}_m \colon R_j^{q.\textnormal{frep}} \to R_{j+m}/IR_{j+m}$ the $m$-th projection map. If $m=0$, we simply write $\Phi^{(j)}=\Phi^{(j)}_0$.
    %We use the same notation if $\{R_i,t_i\}_{i \geq 0}$ is a perfectoid tower.
    \item We set $R_\infty^{q.\textnormal{frep}}\coloneqq \varinjlim \{ R_0^{q.\textnormal{frep}} \xrightarrow{t_0^{q.\textnormal{frep}}} R_1^{q.\textnormal{frep}} \xrightarrow{t_1^{q.\textnormal{frep}}} R_2^{q.\textnormal{frep}} \xrightarrow{t_2^{q.\textnormal{frep}}} \cdots \}$.
    \end{itemize}
  \item Consider a perfectoid tower $\{R_i,t_i\}_{i\geq 0}$ arising from some pair $(R,I)$.
    \begin{itemize}
    \item We set $R_\infty^{s.\flat}\coloneqq \varinjlim \{ R_0^{s.\flat} \xrightarrow{t_0^{s.\flat}} R_1^{s.\flat} \xrightarrow{t_1^{s.\flat}} R_2^{s.\flat} \xrightarrow{t_2^{s.\flat}} \cdots \}$.
    \item We set $R_\infty^\flat \coloneqq (R_\infty)^\flat \cong (\wh{R_\infty})^\flat \cong\footnote{This isomorphism follows from \cite[Lemma 3.55]{INS25}.} \wh{R_\infty^{s.\flat}}$ where $\wh{R_\infty^{s.\flat}}$ is the $I^{s.\flat}$-adic completion of $R_\infty^{s.\flat}$.
    \end{itemize}
  \end{enumerate}
\end{notation}

\begin{lemma}
\label{lem:torisom}
Let $\{R_i,t_i\}_{i\geq 0}$ be a perfectoid tower arising from some pair $(R,I_0)$. Then we have the isomorphisms of (possibly) non-unital rings
\begin{equation}
\label{eq:tor2}
\begin{tikzcd}
(R_\infty^{s.\flat})_{\textrm{$I_0^{s.\flat}$-}\mathrm{tor}} \rar["\cong"] \dar["\cong"'] & 
(R_\infty^\flat)_{\textrm{$I_0^{s.\flat}$-}\mathrm{tor}} \\
(R_\infty)_{\textrm{$I_0$-}\mathrm{tor}} \rar["\cong"] & 
(\wh{R_\infty})_{\textrm{$I_0$-}\mathrm{tor}}
\end{tikzcd}
\end{equation}
\end{lemma}

\begin{proof}
The vertical isomorphism is obtained by \cite[Theorem 3.35 (2)]{INS25}. Moreover, it follows from axiom (g) that $I_0(R_\infty)_{\textrm{$I_0$-}\mathrm{tor}}=(0)$. Then the restriction of the canonical homomorphism $R_\infty\to\wh{R_\infty}$ induces an isomorphism $(R_\infty)_{\textrm{$I_0$-}\mathrm{tor}} \xr{\cong} (\wh{R_\infty})_{\textrm{$I_0$-}\mathrm{tor}}$ by \cite[Lemma 3.16]{INS25}. Since $\{R_i^{s.\flat},t_i^{s.\flat}\}_{i\geq 0}$ is a perfectoid tower arising from $(R^{s.\flat},I_0^{s.\flat})$ by \cite[Proposition 3.41]{INS25}, we can apply the same argument to obtain the isomorphism $(R_\infty^{s.\flat})_{\textrm{$I_0^{s.\flat}$-}\mathrm{tor}} \to (R_\infty^\flat)_{\textrm{$I_0^{s.\flat}$-}\mathrm{tor}}$.
\end{proof}

\begin{remark}
The composition $(R_\infty^\flat)_{\textrm{$I_0^{s.\flat}$-}\mathrm{tor}} \xleftarrow{\cong}
(R_\infty^{s.\flat})_{\textrm{$I_0^{s.\flat}$-}\mathrm{tor}} \xr{\cong}(R_\infty)_{\textrm{$I_0$-}\mathrm{tor}} \xr{\cong} (\wh{R_\infty})_{\textrm{$I_0$-}\mathrm{tor}}$ coincides with the isomorphism induced by $\sharp\colon R_\infty^\flat\to \wh{R_\infty}$ (\cref{thm:sharp} (2)).
\end{remark}

\begin{proposition-definition}
Let $\{R_i,t_i\}_{i\geq 0}$ be a perfectoid tower arising from some pair $(R,I_0)$. Then the following conditions are equivalent.
  \begin{enumerate}
  \item $R_0$ is $I_0$-torsion free.
  \item $R_i$ is $I_0$-torsion free for any $i\geq 0$.
  \item $R_\infty$ is $I_0$-torsion free.
  \item $\wh{R_\infty}$ is $I_0$-torsion free.
  \item $R_0^{s.\flat}$ is $I_0^{s.\flat}$-torsion free.
  \item $R_i^{s.\flat}$ is $I_0^{s.\flat}$-torsion free for any $i\geq 0$.
  \item $R_\infty^{s.\flat}$ is $I_0^{s.\flat}$-torsion free.
  \item $R_\infty^\flat$ is $I_0^{s.\flat}$-torsion free.
  \end{enumerate}
If one of (and hence all) these conditions are satisfied, then we say that the perfectoid tower $\{R_i,t_i\}_{i\geq 0}$ is \emph{$I_0$-torsion free}.
\end{proposition-definition}

\begin{proof}
%By \cite[Theorem 3.35 (2)]{INS25} and \cref{lem:torisom}, we have the equivalences
%\[
%\begin{tikzcd}
%\text{(1)} \rar[Leftrightarrow] \dar[Leftrightarrow] & \text{(2)} \dar[Leftrightarrow] & \text{(3)} \rar[Leftrightarrow] \dar[Leftrightarrow] & \text{(4)} \dar[Leftrightarrow] \\
%\text{(5)} \rar[Leftrightarrow] & \text{(6)} & \text{(7)} \rar[Leftrightarrow] & \text{(8)}.
%\end{tikzcd}
%\]
%
%``(2) $\Rightarrow$ (3)'' and ``(6) $\Rightarrow$ (7)''
%(6) $\Rightarrow$ (7): This implication follows from the exactness of colimits.
%
%(7) $\Rightarrow$ (6): This follows from the fact that every $R_i^{s.\flat}\to R_{i+1}^{s.\flat}$ is injective.
%
By \cite[Theorem 3.35 (2)]{INS25}, the equivalence ``(1) $\Longleftrightarrow$ (5)'' (or ``(2) $\Longleftrightarrow$ (6)'') holds. On the other hand, it follows from \cite[Proposition 3.41]{INS25} that the tilt $\{R_i^{s.\flat},t_i^{s.\flat}\}_{i\geq 0}$ is a perfectoid tower arising from $(R^{s.\flat},I_0^{s.\flat})$. Hence it suffices to show the equivalences from (1) to (4) after replacing $\{R_i,t_i\}_{i\geq 0}$ by $\{R_i^{s.\flat},t_i^{s.\flat}\}_{i\geq 0}$. In particular, we may assume that the maps $t_i$ are injective.

(1) $\Leftrightarrow$ (2): This follows immediately from axiom (g).

(2) $\Rightarrow$ (3): This follows from the exactness of colimits.

(3) $\Rightarrow$ (2): This is clear because the maps $t_i$ are injective.

(3) $\Leftrightarrow$ (4): This follows from \cref{lem:torisom}.
\end{proof}

The following lemma is an application of the results in the previous section:

\begin{lemma}
\label{lem:citower}
Let $\{R_i,t_i\}_{i\geq 0}$ be a tower of rings satisfying (a) and (b) for some pair $(R,I)$, where $I$ is generated by an element $f$. Assume that $R_i$ is $I$-torsion free for every $i\in\N\cup\{\infty\}$.
  \begin{enumerate}
  \item Let $i\geq 0$. If $R_{i+1}$ is integrally closed (resp.\ completely integrally closed) in $R_{i+1}[\frac{1}{f}]$, then $R_i$ is integrally closed (resp.\ completely integrally closed) in $R_i[\frac{1}{f}]$.
  \item The following conditions are equivalent.
    \begin{enumerate}
    \item $R_i$ is completely integrally closed in $R_i[\frac{1}{f}]$ for every $i\geq 0$.
    \item $R_\infty$ is completely integrally closed in $R_\infty[\frac{1}{f}]$.
    \item $\wh{R_\infty}$ is completely integrally closed in $\wh{R_\infty}[\frac{1}{f}]$.
    \end{enumerate}
  \end{enumerate}
\end{lemma}

\begin{proof}
(1) follows from \cref{lem:cartesian,lem:ci cartesian}.
(2) follows from \cref{prop:ci cartesian3}.
\end{proof}

%%%%%%%%%%%%%%%%%%%%%%%%%%%%%%%%%%%%%%%%%%%%%%%%%%%%%%%%%%%
\subsection{Monoidal maps for perfectoid towers}

Let $\{R_i,t_i\}_{i\geq 0}$ be a perfectoid tower arising from some pair $(R,I)$. Then the tilt $R_\infty^\flat$ of the $p$-adically complete ring $\wh{R_\infty}$ is isomorphic to the multiplicative monoid $\varprojlim\limits_{x\mapsto x^p}\wh{R_\infty}$ (\cref{lem:monoidal}). Here we will refine this result for small tilts $R_i^{s.\flat}$. The essential point relies on the following lemma.

\begin{lemma}
\label{lem:inj}
Let $\{R_i,t_i\}_{i\geq 0}$ be a tower of rings satisfying axioms (a) and (b) for some pair $(R,I)$. Assume that the ideal $IR_\infty \subseteq R_\infty$ is finitely generated (e.g., axiom (f) is satisfied). Then for any $i\geq 0$, the ring map
\[
R_i/IR_i \to \wh{R_\infty}/I\wh{R_\infty}
\]
induced by the composition $R_i \to R_\infty \to \wh{R_\infty}$ is injective.
\end{lemma}

\begin{proof}
$R_i/IR_i \to \wh{R_\infty}/I\wh{R_\infty}$ factors as
\[
R_i/IR_i \to R_\infty/IR_\infty \to \wh{R_\infty}/I\wh{R_\infty}.
\]
The first arrow is induced from the injections $R_i/IR_i \hookrightarrow R_j/IR_j$ ($j\geq i$), hence is injective. The second arrow is an isomorphism by the assumption that $IR_\infty$ is finitely generated.
\end{proof}

In particular, $R_i/IR_i$ is isomorphic to its image in $\wh{R_\infty}/I\wh{R_\infty}$. We will use the following notation:

\begin{notation}
For a tower of rings $\{R_i,t_i\}_{i\geq 0}$ satisfying (a) for some pair $(R,I)$ and for every $i\geq 0$, let
\[
R_i+I\wh{R_\infty}\subseteq \wh{R_\infty}
\]
denote the sum (as additive subgroups of $\wh{R_\infty}$) of the image of $R_i$ in $\wh{R_\infty}$ and the ideal $I\wh{R_\infty}$ of $\wh{R_\infty}$. Note that $R_i+I\wh{R_\infty}$ is a subring of $\wh{R_\infty}$, which coincides with the $R_i$-subalgebra of $\wh{R_\infty}$ generated by the subset $I\wh{R_\infty}\subseteq \wh{R_\infty}$.
\end{notation}

Under this notation, suppose that $\{R_i,t_i\}_{i\geq 0}$ is a purely inseparable tower arising from $(R,I)$. Then the injection $R_i/IR_i \hookrightarrow \wh{R_\infty}/I\wh{R_\infty}$ of \cref{lem:inj} induces an isomorphism
\begin{equation}
\label{eq:isom alpha}
R_i/IR_i \xr{\cong} (R_i+I\wh{R_\infty})/I\wh{R_\infty}.
\end{equation}
Note that we have a commutative diagram
\[
\begin{tikzcd}
 R_i \rar \dar["t_i"'] & R_i + I\widehat{R_\infty} \dar \\
 R_{i+1} \rar & R_{i+1} + I\widehat{R_\infty},
\end{tikzcd}
\]
where the right vertical arrow is the inclusion $R_i + I\wh{R_{\infty}} \subseteq R_{i+1}+I\widehat{R_\infty}$.
Moreover, the axiom (c) implies that the $p$-th power map on $\wh{R_\infty}$ induces a multiplicative map between subrings
\[
R_{i+1}+I\wh{R_\infty} \to R_i+I\wh{R_\infty}
\]
for any $i\geq 0$. This motivates the following definition.

%Easiest example:
%\[
%\mathbb{Z}_p + p\wh{\mathbb{Z}_p[p^{1/p^\infty}]} \subset \wh{\mathbb{Z}_p[p^{1/p^\infty}]}
%\]

\begin{definition}[Multiplicative small tilt]
Let $\{R_i,t_i\}_{i\geq 0}$ be a purely inseparable tower arising from some pair $(R,I)$ with $IR_\infty$ being finitely generated. For every $i\geq 0$, we endow $R_i+I\wh{R_\infty}$ with the subspace topology of $\wh{R_\infty}$.
For any $j\geq 0$, we define the \emph{$j$-th multiplicative small tilt of $\{R_i,t_i\}_{i\geq 0}$ associated to $(R,I)$} as the limit of topological monoids
\[
R_j^{\mult.\flat} \coloneqq \varprojlim\{\cdots \to R_{j+i+1}+I\wh{R_\infty} \xr{x\mapsto x^p} R_{j+i}+I\wh{R_\infty} \to\cdots \xr{x\mapsto x^p} R_j+I\wh{R_\infty}\}.
\]
\end{definition}

Let us list basic topological properties of the rings $R_i+I\wh{R_\infty}$.

\begin{lemma}
\label{lem:natural}
Let $\{R_i,t_i\}_{i\geq 0}$ be a tower of rings satisfying (a) for some pair $(R,I)$ with the ideal $I\wh{R_\infty}\subseteq \wh{R_\infty}$ is finitely generated. Let $i\geq 0$.
  \begin{enumerate}
  \item $R_i+I\wh{R_\infty}$ is an open subring of $\wh{R_\infty}$.
  \item $I\wh{R_\infty}$ is an ideal of $R_i+I\wh{R_\infty}$, and the subspace topology on  $R_i+I\wh{R_\infty}\subseteq \wh{R_\infty}$ coincides with the $I\wh{R_\infty}$-adic topology as well as the $I(R_i+I\wh{R_\infty})$-adic topology.
  \item $R_i+I\wh{R_\infty}$ is $I\wh{R_\infty}$-adically complete.
  \end{enumerate}
\end{lemma}

\begin{proof}
The assertion (1) follows from the fact that $R_i+I\wh{R_\infty}$ contains $I\wh{R_\infty}$.
The assertion (2) follows from $I^2\wh{R_\infty}\subseteq I(R_i+I\wh{R_\infty}) \subseteq I\wh{R_\infty}$.
Finally, it follows from (1) that $R_i+I\wh{R_\infty}$ is closed in $\wh{R_\infty}$. Hence it is complete with respect to the subspace topology, which is $I\wh{R_\infty}$-adic by (2).
Hence the assertion (3) holds.
\end{proof}

\begin{remark}
The ring $R_i+I\wh{R_\infty}$ is rarely Noetherian. Indeed, we have a commutative diagram of rings
\[
\begin{tikzcd}
R_i/IR_i \rar \ar[rd,"\cong"'] & (R_i+I\wh{R_\infty})/I(R_i+I\wh{R_\infty}) \dar \\
 & (R_i+I\wh{R_\infty})/I\wh{R_\infty}.
\end{tikzcd}
\]
This implies that there exists a decomposition of $R_i$-modules
\[
\frac{R_i+I\wh{R_\infty}}{I(R_i+I\wh{R_\infty})} \cong (R_i/IR_i) \oplus \frac{I\wh{R_\infty}}{I(R_i+I\wh{R_\infty})}.
\]
Hence, if $I$ is generated by an element $f$ such that $R_i$ is $f$-torsion free for every $i\in\N\cup\{\infty\}$, then $(R_i+I\wh{R_\infty})/I(R_i+\wh{R_\infty})$ has a direct summand isomorphic to the cokernel of the inclusion $R_i+I\wh{R_\infty}\hookrightarrow \wh{R_\infty}$, which is isomorphic to that of the canonical injection $R_i/IR_i \hookrightarrow R_\infty^\flat/IR_\infty^\flat$.
Consider, for example, the perfectoid tower associated to the regular local ring $\mathbb{Z}_p$. Then the cokerel of $\F_p[[T]] \to  \wh{\F_p[T^{\frac{1}{p^\infty}}]}$ is not finite dimensional over $\F_p$, and so is $(\mathbb{Z}_p+p\wh{\mathbb{Z}_p[p^{\frac{1}{p^\infty}}]})/p(\mathbb{Z}_p+p\wh{\mathbb{Z}_p[p^{\frac{1}{p^\infty}}]})$. Consequently, $\mathbb{Z}_p+p\wh{\mathbb{Z}_p[p^{\frac{1}{p^\infty}}]}$ is not of finite type over $\mathbb{Z}_p$.
\end{remark}

%\begin{proposition}
%If $\{R_i,t_i\}_{i\geq 0}$ is a perfectoid tower arising from some pair $(R,I_0)$, then the rings $R_i+I_0\wh{R_\infty}$ together with canonical injections form a perfectoid tower arising from $(R+I_0\wh{R_\infty}, I_0\wh{R_\infty})$. Moreover, the tower satisfies
%\[
%\left(\varinjlim_{i\geq 0}(R_i+I_0\wh{R_\infty})\right)^\wedge = \wh{R_\infty},\quad (R_i+I_0\wh{R_\infty})/I_0\wh{R_\infty} \cong R_i/I_0R_i
%\]
%\end{proposition}

\begin{proposition}
\label{prop:refined monoidal}
Let $\{R_i,t_i\}_{i\geq 0}$ be a purely inseparable tower arising from some pair $(R,I)$ with the ideal $IR_\infty$ of $R_\infty$ being finitely generated. Then for any $j\geq 0$, the composition of the canonical maps
\[
R_i+I\wh{R_\infty} \to (R_i+I\wh{R_\infty})/I\wh{R_\infty} \xleftarrow{\cong} R_i/IR_i\quad (i\geq j)
\]
induces an isomorphism of topological monoids
\[
R_j^{\mult.\flat} \xr{\cong} R_j^{q.\textnormal{frep}}.
\]
\end{proposition}

\begin{proof}
We have the commutative diagram
\begin{equation}
\label{eq:mflat sflat inf}
\begin{tikzcd}
\varprojlim\limits_{x\mapsto x^p}\wh{R_\infty} \rar["\cong"] & (\widehat{R_\infty})^\flat \\
R_j^{\mult.\flat} \rar \uar[hookrightarrow] & R_j^{q.\textnormal{frep}}, \uar[hookrightarrow]
\end{tikzcd}
\end{equation}
where the left-hand vertical arrow is the canonical injection, the right-hand vertical arrow is the injection defined as the limit of $\{\ol{t_{j+i}}\}_{i\geq 0}$, and the upper horizontal arrow is an isomorphism because $\wh{R_\infty}$ is $p$-adically complete.

Due to the commutative diagram \eqref{eq:mflat sflat inf}, it suffices to show that the inverse \eqref{eq:inverse} of the upper horizontal arrow maps $R_j^{q.\textnormal{frep}}$ into $R_j^{\mult.\flat}$. But this immediately follows from the fact that $R_i+I\wh{R_\infty}$ is closed in $\wh{R_\infty}$ (\cref{lem:natural} (1)).
\end{proof}

Once we have shown \cref{prop:refined monoidal}, it is natural to define refined Fontaine's maps as follows:

\begin{definition}[Refined Fontaine's map]
\label{def:refined monoidal}
Let $\{R_i,t_i\}_{i\geq 0}$ be a purely inseparable tower arising from some pair $(R,I)$ with $IR_\infty$ being finitely generated. For any $j\geq 0$, we define the map of topological monoids $\sharp^{(j)}\colon R_j^{q.\textnormal{frep}} \to R_{j+m}+I\wh{R_\infty}$ as the composition
\[
\sharp^{(j)}\colon R_j^{q.\textnormal{frep}} \xleftarrow{\cong} R_j^{\mult.\flat} \xr{\pr_0} R_j+I\wh{R_\infty},
\]
where the first isomorphism arises from \cref{prop:refined monoidal}, and the second map is the $0$-th projection.
\end{definition}

\begin{remark}
Although the refined monoidal map $\sharp^{(j)}$ is nothing more than the restriction of the ordinary monoidal map $\sharp\colon R_\infty^\flat \to \wh{R_\infty}$ (see the commutative diagram \eqref{eq:mflat sflat inf}), our construction confirms that $\sharp$ maps $R_j^{q.\textnormal{frep}}$ into $R_j+I\wh{R_\infty}$. Here is an alternative proof of this result. We use the so-called the \emph{theta map} $\theta\colon W(R_\infty^\flat)\to \wh{R_\infty}$, which is well-defined because of the $p$-adically completeness of $\wh{R_\infty}$; see \cite[\S3]{BMS1} for its construction. Indeed, we have the commutative diagram
\[
\begin{tikzcd}
R_\infty^\flat \rar["{[\cdot]}"] \ar[rr,bend left,"\sharp"] \ar[rd,equal] & W(R_\infty^\flat) \dar["\omega_0"] \rar["\theta"] & \wh{R_\infty} \dar["\ \mathrm{mod}\ I\wh{R_\infty}"] \\
R_j^{q.\textnormal{frep}} \uar[hookrightarrow] \ar[rd,equal] & R_\infty^\flat \rar["\pr_0"'] & \wh{R_\infty}/I\wh{R_\infty} \\
 & R_j^{q.\textnormal{frep}} \uar[hookrightarrow] \rar["\pr_0"'] & R_j/IR_j \uar[hookrightarrow]
\end{tikzcd}
\]
where $[\cdot]$ is the Teichm\"{u}ller map and $\omega_0$ is the $0$-th component map. Hence $\sharp$ maps $R_j^{q.\textnormal{frep}}$ to the image of $R_j/IR_j \hookrightarrow \wh{R_\infty}/I\wh{R_\infty}$, as claimed. 
\end{remark}

\begin{remark}
By definition, we have the following commutative diagram
\begin{equation}
\label{eq:sharpPhi}
\begin{tikzcd}
R_j^{q.\textnormal{frep}} \rar["\sharp^{(j)}"] \ar[rd,"\Phi^{(j)}"'] & R_j+I\wh{R_\infty} \dar \\
 & R_j/IR_j.
\end{tikzcd}
\end{equation}
\end{remark}

In the rest of this section, we provide an analogue of \cref{thm:sharp}.

\begin{proposition}
Let $\{R_i,t_i\}_{i\geq 0}$ be a purely inseparable tower arising from some pair $(R,I)$ with $IR_\infty$ being finitely generated. Fix $j\geq 0$. If $R_j$ is $I$-adically henselian, then the map
\[
\mathrm{Idem}(R_j) \to \mathrm{Idem}(R_j^{q.\textnormal{frep}});\quad e\mapsto (e,e,e,\ldots)
\]
is bijective. In paritcular, $R_j$ is connected if and only if so is $R_j^{q.\textnormal{frep}}$.
\end{proposition}

\begin{proof}
The assertion follows from the commutative diagram
\[
\begin{tikzcd}
\mathrm{Idem}(R_j) \rar \dar["\sim"' sloped] & \mathrm{Idem}(R_j^{q.\textnormal{frep}}) \dar["\sharp^{(j)}","\sim"' sloped] \\
\mathrm{Idem}(R_j/IR_j) & \mathrm{Idem}(R_j+I\wh{R_\infty}), \lar["\sim"]
\end{tikzcd}
\]
where the lower horizontal map is bijective by \cref{lem:natural} (3) and \cite[Corollary 3.4.16]{Ford}. Moreover, the left-hand vertical map is bijective because $R_j$ is $I$-adically henselian (\cite[\href{https://stacks.math.columbia.edu/tag/09XI}{Tag 09XI}]{stacks-project}).
\end{proof}

\begin{lemma}
\label{lem:sharpf}
Let $\{R_i,t_i\}_{i\geq 0}$ be a perfectoid tower arising from some pair $(R,I_0)$. Let $\{I_i\}_{i\geq 0}$ be the system of perfectoid pillars (\cite[Definition 3.27]{INS25}). Then we have a sequence of principal ideals $\{I_i^{s.\flat}\}_{i\geq 0}$ as in \cite[Definition 3.29 (2)]{INS25}. Fix $j\geq 0$.
If $f_j^{s.\flat}\in I_j^{s.\flat}$ denotes a generator of $I_j^{s.\flat}$, then $\sharp^{(j)}(f_j^{s.\flat})$ generates the ideal $I_{j}(R_{j}+I_0\wh{R_\infty})$. In particular, if $\wh{R_\infty}$ is $I_0$-torsion free, then
\[
\sharp^{(j)}(f_{j}^{s.\flat}) = f_{j}u_{j}
\]
for some unit $u_{j}\in\wh{R_\infty}^\times$, where $f_{j}$ denotes a generator of $I_{j}$.
\end{lemma}

\begin{proof}
By \cite[Lemma 3.30]{INS25}, we have $\Phi^{(j)}(I_j^{s.\flat}) = I_{j}({R_{j}}/I_0R_{j})$. Here $\Phi^{(j)}_i(I_j^{s.\flat})$ is generated by $\sharp^{(j)}(f_j^{s.\flat})\ \mathrm{mod}\ I_0\wh{R_\infty} \in {R_{j}}/I_0R_{j}$ by \eqref{eq:sharpPhi}. Thus we obtain the assertion by pulling back to $R_{j}+I_0\wh{R_\infty}$.
\end{proof}

\begin{proposition}
\label{prop:sharpi2}
Let $\{R_i,t_i\}_{i\geq 0}$ be a perfectoid tower arising from some pair $(R,I_0)$. For every $j\geq 0$, $\sharp^{(j)}\colon R_j^{s.\flat}\to R_j+I_0\wh{R_\infty}$ induces an isomorphism of rings
\begin{equation}
\label{eq:sharp1}
R_j^{s.\flat}/I_0^{s.\flat}R_j^{s.\flat} \xr{\cong} R_j/I_0R_j
\end{equation}
and an isomorphism of (possibly) non-unital rings
\begin{equation}
\label{eq:sharp2}
(R_j^{s.\flat})_{\textrm{$I_0^{s.\flat}$-}\mathrm{tor}}\xr{\cong} (R_j)_{\textrm{$I_0$-}\mathrm{tor}}
\end{equation}
\end{proposition}

\begin{proof}
\eqref{eq:sharp1} follows from \cite[Lemma 3.39]{INS25} combined with \eqref{eq:sharpPhi}. To show \eqref{eq:sharp2}, note that $\sharp$ induces an isomorphism $(R_\infty^\flat)_{\textrm{$I_0^{s.\flat}$-}\mathrm{tor}} \xr{\cong} (\wh{R_\infty})_{\textrm{$I_0$-}\mathrm{tor}}$ by \cref{lem:sharpf}. Combining this with \cite[Theorem 3.35 (2)]{INS25}, we have the assertion.
\[
\begin{tikzcd}[row sep=small]
(R_\infty^\flat)_{\textrm{$I_0^{s.\flat}$-}\mathrm{tor}} \rar["\sharp","\cong"'] & (\wh{R_\infty})_{\textrm{$I_0$-}\mathrm{tor}} \\
 & (R_j+I_0\wh{R_\infty})_{\textrm{$I_0$-}\mathrm{tor}} \uar[hookrightarrow] \\
(R_j^{s.\flat})_{\textrm{$I_0^{s.\flat}$-}\mathrm{tor}} \rar["\cong"] \ar[uu,hookrightarrow] \ar[ru,"\sharp^{(j)}"] & (R_j)_{\textrm{$I_0$-}\mathrm{tor}} \uar
\end{tikzcd}
\]
\end{proof}

%\begin{equation}
%\label{eq:tor2}
%\begin{tikzcd}
%(R_\infty^{s.\flat})_{\textrm{$I_0^{s.\flat}$-}\mathrm{tor}} \rar["\cong"] \dar["\cong"'] & (R_\infty^\flat)_{\textrm{$I_0^{s.\flat}$-}\mathrm{tor}} \dar["\cong"',"\sharp"] \\
%(R_\infty)_{\textrm{$I_0$-}\mathrm{tor}} \rar["\cong"] & (\wh{R_\infty})_{\textrm{$I_0$-}\mathrm{tor}}
%\end{tikzcd}
%\end{equation}

\begin{theorem}
\label{tiltingintegral}
Let $\{R_i,t_i\}_{i\geq 0}$ be an $I_0$-torsion free perfectoid tower arising from some pair $(R,I_0)$. Let $I_0=(f_0)$ and $I_0^{s.\flat}=(f_0^{s.\flat})$. If $R_i$ is completely integrally closed in $R_i[\frac{1}{f_0}]$ for every $i\geq 0$, then $R_i^{s.\flat}$ is completely integrally closed in $R_i^{s.\flat}[\frac{1}{f_0^{s.\flat}}]$ for every $i\geq 0$.
\end{theorem}

\begin{proof}
This follows from \cref{lem:sharpf}, \cref{lem:citower}, and \cref{thm:sharp} (3).
\end{proof}

%%%%%%%%%%%%%%%%%%%%%%%%%%%%%%%%%%%%%%%%%%%%%%%%%%%%%%%%%%%%%%%%%%%%%%%%
%%%%%%%%%%%%%%%%%%%%%%%%%%%%%%%%%%%%%%%%%%%%%%%%%%%%%%%%%%%%%%%%%%%%%%%
\section{Perfectoid towers arising from ramification theory}
\label{s:4}

In this section, we provide a construction of perfectoid towers arising from the context of ramification theory. Such a construction is originally due to Faltings \cite{Fa88} and \cite{Fa02}. Here, we follow the method by Andreatta \cite{An06}. 

%in particular, \Cref{ss:d} is devoted to the most difficult part of the proof of \cref{thm:B}.
%We utilize certain ramified towers via the almost purity theorem to construct a perfectoid tower, which is originally due to Faltings \cite{Fa88} and \cite{Fa02}. Here, we follow the method by Andreatta \cite{An06}. 

%%%%%%%%%%%%%%%%%%%%%%%%%%%%%%%%%%%%%%%%%%%%%%%%%%%%%%%
\subsection{Construction}

We begin with the following proposition concerning finiteness of integral closures.

\begin{proposition}
\label{prop:Sn normal}
Let $R\to R'$ be a flat homomorphism of normal rings.
Let $R\to S$ be an integral ring homomorphism such that there exists $g\in R$ such that both $R$ and $S$ are $g$-torsion free and the induced map $R[\frac{1}{g}]\to S[\frac{1}{g}]$ is weakly \'{e}tale.\footnote{We say that a ring map $A\to B$ is \emph{weakly \'{e}tale} if both $A\to B$ and the multiplication map $B\otimes_AB\to B$ are flat.} % (\cite[Definitiion 3.1.1]{GR03}).}
%\[
%\begin{tikzcd}
%R \rar \dar & S \dar \\
%R' \rar & R'\otimes_RS
%\end{tikzcd}
%\]
%\footnote{The normality of $S$ is required just for the equality $S=S'$ in the case $R=R'$.}, where both
Then the following assertions hold.
  \begin{enumerate}
  \item $R'\otimes_RS$ is $g$-torsion free.
  \item $(R'\otimes_RS)[\frac{1}{g}]$ is a normal ring.
  \item The integral closure $S'$ of $R'\otimes_RS$ in the total ring of fractions $Q(R'\otimes_RS)$ coincides with that of $R'$ in $(R'\otimes_RS)[\frac{1}{g}]$.
  %\[
  %R' \hookrightarrow R'\otimes_RS \hookrightarrow S' \hookrightarrow (R'\otimes_RS)[\tfrac{1}{g}] \hookrightarrow Q(R'\otimes_RS)
  %\]
  \item Assume that $R\to S$ is finite and $R'$ is Nagata (e.g., complete Noetherian local). Then $R'\to S'$ is finite. Hence, in particular, $S'$ is Noetherian.
  \end{enumerate}
\end{proposition}

\begin{proof}
(1) follows from the fact that $S$ is $g$-torsion free and the flatness of $R\to R'$.

(2) Since $R[\frac{1}{g}]\to S[\frac{1}{g}]$ is weakly \'{e}tale, its base change $R'[\tfrac{1}{g}] \to (R'\otimes_RS)[\tfrac{1}{g}]$ is also weakly \'{e}tale. Then $(R'\otimes_RS)[\frac{1}{g}]$ is normal because so is $R'[\frac{1}{g}]$ (\cite[\href{https://stacks.math.columbia.edu/tag/0950}{Tag 0950}]{stacks-project}).

(3) %Since $R_n[\frac{1}{f}]$ is a normal domain, \eqref{eq:fet} implies that $(R_n\otimes_RS)[\frac{1}{f}]$ is integrally closed in $Q(R_n[\frac{1}{f}])\otimes_{R_n[\frac{1}{f}]}(R_n\otimes_RS)[\frac{1}{f}]$ (\cite[\href{https://stacks.math.columbia.edu/tag/092W}{Tag 092W}]{stacks-project}). But
%\[
%Q(R_n[\tfrac{1}{f}])\otimes_{R_n[\tfrac{1}{f}]}(R_n\otimes_RS)[\tfrac{1}{f}] \cong Q\left((R_n\otimes_RS)[\tfrac{1}{f}]\right) = Q(R_n\otimes_RS)
%\]
%by \cref{lem:total ring of frac}, and so we obtain the assertion.
By (2), $S'$ coincides with the integral closure of $R'\otimes_RS$ in $(R'\otimes_RS)[\frac{1}{g}]$. Hence, it suffices to check that $R'\otimes_RS$ is integral over $R'$. But this is clear, since $R'\to R'\otimes_RS$ is the base change of the integral map $R\to S$.

(4) By assumption, $(R'\otimes_RS)[\frac{1}{g}]$ is essentially of finite type over the Nagata ring $R'$. Thus \cite[\href{https://stacks.math.columbia.edu/tag/03GH}{Tag 03GH}]{stacks-project} together with (3) shows that $R'\to S'$ is finite.
%(6) By (2), (4) and (5), $S_n$ is a Noetherian reduced ring. Hence $S_n$ is normal if and only if it is integrally closed in its total ring of fractions (\cite[\href{https://stacks.math.columbia.edu/tag/030C}{Tag 030C}]{stacks-project}), which follows immediate from (4).
\end{proof}

We fix notations.

\begin{notation}
In the rest of this section, we use the following notation.
  \begin{itemize}
  \item $R$ : a $d$-dimensional unramified complete regular local ring of mixed characteristic $(0,p)$ whose residue field $k$ is perfect.
  \item $\{R_n\}_{n\geq 0}$ : the associated perfectoid tower arising from $(R,(p))$ (see \Cref{eg:RLRTower}).
  \item $R\hookrightarrow S$ : a finite extension of normal domains such that $R[\frac{1}{p}]\to S[\frac{1}{p}]$ is \'{e}tale.
  \item $S_n$ ($n\geq 0$) : the integral closure of $R_n\otimes_RS$ in $Q(R_n\otimes_RS)$. By \cref{prop:Sn normal} (3) (4), $S_n$ is the integral closure of $R_n$ in $(R_n\otimes_RS)[\frac{1}{p}]$, and $R_n\to S_n$ is finite such that $R_n[\frac{1}{p}]\to S_n[\frac{1}{p}]=(R_n\otimes_RS)[\frac{1}{p}]$ is \'{e}tale.
  \end{itemize}
\end{notation}

Our aim here is to show that the tower $\{S_n\}_{n\geq N}$ is a perfectoid tower for some $N\geq 0$. The key ingredient is the following proposition, which relies on ramification theory:

\begin{proposition}[Faltings]
\label{FaltingsLemma}
There exists a decreasing sequence $\{\delta_n\}_{n \ge 0}$ of positive rational numbers satisfying the following conditions.
  \begin{enumerate}
  \item $p^{\delta_n} \in W(k)[p^{\frac{1}{p^n}}]$.
  \item $\displaystyle \lim_{n \to \infty} \delta_n=0$.
  \item For every $n\geq 0$, the cokernel of $R_{n+1} \otimes_{R_n} S_n \to S_{n+1}$ is annihilated by $p^{\delta_n}$.
  \end{enumerate}
\end{proposition}

For the proof, see \cite[3.6. Proposition]{An06}.

\begin{remark}
\label{rem:deltaN}
The condition $p^{\delta_n}\in W(k)[p^{\frac{1}{p^n}}]$ implies that $\delta_np^n\in\mathbb{Z}_{\geq 1}$. Indeed, since $(W(k)[p^{\frac{1}{p^n}}], (p^{\frac{1}{p^n}}))$ is a DVR, we can write $p^{\delta_n} = p^{\frac{m}{p^n}}u_n$ for some integer $m\geq 0$ and unit $u_n\in W(k)[p^{\frac{1}{p^n}}]^\times$. Then $p^{\delta_np^n} = p^m u_n^{p^n}$ and $u_n^{p^n}\in W(k)^\times$. Thus we get $\delta_np^n=m\in\mathbb{Z}$ because $(W(k),(p))$ is a DVR.
\end{remark}

For a ring $A$ and an integer $n>0$, we denote by $A^n$ the image of the $n$-th power map $(-)^n:A \to  A$. Hence $A^n$ is a multiplicative submonoid of $A$. The next result is due to Andreatta (see \cite[3.7 Corollary]{An06}). For ease of reference, we provide the proof.

\begin{corollary}
\label{p-Frobenius}
There exist a rational $\varepsilon \in (0,1) \cap \mathbb{Q}$ and an integer $N \ge 0$ such that $p^\varepsilon \in W(k)[p^{\frac{1}{p^N}}]$ and
$$
(S_{n+1})^p \subseteq S_n+p^\varepsilon S_{n+1}
$$
for all $n \ge N$.
\end{corollary}

\begin{proof}
Notice that $(R_{n+1})^p+p^{\frac{1}{p}}R_{n+1}=R_n+p^{\frac{1}{p}}R_{n+1}$, which follows from the surjectivity of the Frobenius projection $R_{n+1}/pR_{n+1}\to R_n/pR_n$ (modulo $p^{\frac{1}{p}}$). Let $\{\delta_n\}_{n \ge 0}$ be the sequence taken as in \Cref{FaltingsLemma}. We have
\begin{align*}
p^{\delta_np}(S_{n+1})^p &\subseteq (R_{n+1}\otimes_{R_n}S_n)^p \\
&\subseteq (R_{n+1})^p\otimes_{R_n}S_n + p(R_{n+1}\otimes_{R_n}S_n) \\
&\subseteq (R_n+p^{\frac{1}{p}}R_{n+1})\otimes_{R_n}S_n + p(R_{n+1}\otimes_{R_n}S_n) \\
&\subseteq S_n + p^{\frac{1}{p}}(R_{n+1}\otimes_{R_n}S_n) 
\subseteq S_n + p^{\frac{1}{p}}S_{n+1}.
\end{align*}
Hence
\[
(S_{n+1})^p \subseteq \frac{1}{p^{\delta_np}}S_n + p^{\frac{1-\delta_np^2}{p}}S_{n+1} = \left(\frac{1}{p^{\delta_np}}S_n + p^{\frac{1-\delta_np^2}{p}}S_{n+1}\right)\cap S_{n+1}
\]
Since $p^n\delta_n \geq 1$ by \cref{rem:deltaN}, we have $p^{\frac{1-\delta_np^2}{p}} \in S_{n+1}$ for all $n\geq 3$. Thus
\[
(S_{n+1})^p \subseteq \frac{1}{p^{\delta_np}}S_n \cap S_{n+1} + p^{\frac{1-\delta_np^2}{p}}S_{n+1} \quad (n\geq 3).
\]
Take $x \in \frac{1}{p^{\delta_np}} S_n \cap S_{n+1}$. Since $S_n \hookrightarrow S_{n+1}$ is integral, we see that $x$ is integral over $S_n$. The normality of $S_n$ implies that $x \in S_n$ and we get
\begin{equation}
\label{p-powermap}
(S_{n+1})^p \subseteq S_n+p^{\frac{1-\delta_np^2}{p}} S_{n+1}.
\end{equation}
On the other hand, by \Cref{FaltingsLemma}, we get
$$
\lim_{n \to \infty} \frac{1-\delta_n p^2}{p}=\frac{1}{p}.
$$
So fix $N \gg 0$ so that $\varepsilon\coloneqq (1-\delta_Np^2)/p$ belongs to $(0,1) \cap \mathbb{Q}$. Then $(\ref{p-powermap})$ yields
$$
(S_{n+1})^p \subseteq S_n+p^{\frac{1-\delta_np^2}{p}} S_{n+1} \subseteq S_n+p^\varepsilon S_{n+1}
$$
for all $n \ge N$.
Finally, since $\varepsilon p^N = p^{N-1}-\delta_N p^{N+1} \in \mathbb{Z}$ by \cref{rem:deltaN}, it follows that $p^\varepsilon = (p^{\frac{1}{p^N}})^{\varepsilon p^N} \in W(k)[p^{\frac{1}{p^N}}]$.
\end{proof}

Now our main theorem of this section is stated as follows.

\begin{theorem}
\label{thm:Sn perfectoid}
There exists some $N' \geq N$ such that $\{S_n\}_{n\geq N'}$ is a perfectoid tower arising from $(S_{N'},(p^\varepsilon))$.
\end{theorem}

It is not difficult to verify the axioms except (d): 

\begin{claim}
$\{S_n\}_{n\geq N}$ is a purely inseparable tower arising from $(S_N,(p^\varepsilon))$. Moreover, it satisfies the axioms {\rm (e)} $\sim$ {\rm (g)} in \cref{def:perfectoid tower}.
\end{claim}

\begin{proof}
(a) is trivial.

(b) follows from \cref{lem:injmod}, since each $S_n$ is a $p$-torsion free normal ring.

(c) is nothing but \cref{p-Frobenius}.

(e) Since $R$ is $p$-adically Zariskian and $R\to S$ is finite, $S$ is also $p$-adically Zariskian (\cref{lem:e}). For every $n\geq 0$, since $S\to S_n$ is finite, $S_n$ is also $p$-adically Zariskian (\cref{lem:e}). This is equivalent to $S_n$ being $p^\varepsilon$-adically Zariskian.

(f) Set $I_1\coloneqq p^{\frac{\varepsilon}{p}} S_{N+1}$. Then the condition (f-1) holds. Moreover, the condition (f-2) follows from \cref{lem:f-2} since every $S_n$ is a $p$-torsion free Noetherian normal ring.

(g) Since $S_n$ is $p$-torsion free for every $n\geq N$, (g) is satisfied automatically.
\end{proof}

We give the proof of axiom (d) in the subsequent \Cref{ss:d}. Before that, let us mention one application. Let
\[
p^\flat \coloneqq (p,p^{\frac{1}{p}},p^{\frac{1}{p^2}},\ldots) \in R_0^{s.\flat}
\]
and, by a slight abuse of notations, we denote by the same $p^\flat$ the image of it in $R_n^{s.\flat}$, $S_n^{s.\flat}$, $R_\infty^{s.\flat}$, $\wh{R_\infty^{s.\flat}}\cong R_\infty^\flat$, etc.

\begin{corollary}
\label{cor:B}
Keep the notation as above. Then for any $n\geq N'$ the following hold.
  \begin{enumerate}
  \item $R_n^{s.\flat}\to S_n^{s.\flat}$ is a finite extension.
  \item $R_n^{s.\flat}[\frac{1}{p^\flat}] \to S_n^{s.\flat}[\frac{1}{p^\flat}]$ is \'{e}tale.
  \item $S_n^{s.\flat}$ is a normal ring.
  \end{enumerate}
\end{corollary}

\begin{proof}
(1) follows from \cite[Lemma 3.39]{INS25} and topological Nakayama's lemma \cite[Theorem 8.4]{Mat86}.

(2) It suffices to show that the base change
\[
R_\infty^{s.\flat}[\tfrac{1}{p^\flat}] \to S_\infty^{s.\flat}[\tfrac{1}{p^\flat}]
\]
of $R_n^{s.\flat}[\frac{1}{p^\flat}] \to S_n^{s.\flat}[\frac{1}{p^\flat}]$ along the faithfully flat map $R_n^{s.\flat}[\frac{1}{p^\flat}] \to R_\infty^{s.\flat}[\frac{1}{p^\flat}]$ is finite \'{e}tale. By Witt-perfect almost purity theorem \cite[Theorem 5.9]{NS18} together with \cite[Theorem 5.7]{NS18}, the $p$-adically completion $\wh{R_\infty}\to\wh{S_\infty}$ is $(p)^{\frac{1}{p^\infty}}$-almost finite \'{e}tale. But $\wh{R_\infty}\to\wh{S_\infty}$ tilts to $\wh{R_\infty^{s.\flat}}\to\wh{S_\infty^{s.\flat}}$, which is $(p^\flat)^{\frac{1}{p^\infty}}$-almost finite \'{e}tale by perfectoid-\'{e}tale correspondence (see \cite[Theorem 7.4.5]{SW}). This implies what we wanted to prove.

(3) It follows from Theorem \ref{tiltingintegral} that $S_n^{s.\flat}$ is integrally closed in $S_n^{s.\flat}[\frac{1}{p^\flat}]$. But we deduce from (2) that $S_n^{s.\flat}[\frac{1}{p^\flat}]$ is normal (\cite[\href{https://stacks.math.columbia.edu/tag/0950}{Tag 0950}]{stacks-project}). Hence $S_n^{s.\flat}$ is normal, as desired.
\end{proof}

%%%%%%%%%%%%%%%%%%%%%%%%%%%%%%%%%%%%%%%%%%%%%%%%%%%%%%%
\subsection{Proof of the surjectivity of Frobenius projections}
\label{ss:d}

It remains to verify (d). 
The proof is essentially established in \cite[5.1 Theorem]{An06}; we provide a detailed proof for the reader's convenience.
We first need the following result, which again relies on ramification theory:

\begin{theorem}[{\cite[3.8. Theorem]{An06}}]
\label{thm:3.8}
The sequence $\{p^n\delta_n\}_{n\geq 0}$ is bounded. Hence there exists a constant $c=c(S)$ such that $p^n\delta_n\leq c$ for all $n\geq 0$.
\end{theorem}

We prepare a few lemmas. The following is a corollary of \cref{thm:3.8}.

\begin{corollary}[{\cite[3.10. Corollary]{An06}}]
\label{cor:3.10}
Let $n\geq N$ be an integer.
  \begin{enumerate}
  \item The cokernel of $R_{n+1}\otimes_{R_n}S_n\hookrightarrow S_{n+1}$ is annihilated by $p^{\frac{c(S)}{p^n}}$.
  \item The cokernel of $R_\infty\otimes_{R_n}S_n\hookrightarrow S_\infty$ is annihilated by $p^{\frac{c(S)p}{p^n(p-1)}}$.
  \end{enumerate}
\end{corollary}

\begin{proof}
(1) follows immediately from \cref{FaltingsLemma} and the inequality $\delta_n\leq \frac{c(S)}{p^n}$.

(2) For each $m\geq 0$, let $C_{n,m}\coloneqq \Cok(R_{n+m}\otimes_{R_n}S_n\hookrightarrow S_{n+m})$. Then we have the commutative diagram with exact rows and columns (where $C$ is defined as cokernels)
\[
\begin{tikzcd}
R_{n+2} \otimes_{R_{n+1}}(R_{n+1}\otimes_{R_n}S_n) \rar \dar["\cong"'] & R_{n+2} \otimes_{R_{n+1}}S_{n+1} \rar \dar & R_{n+2}\otimes_{R_{n+1}} C_{n,1} \rar \dar & 0 \\
R_{n+2}\otimes_{R_n}S_n \rar & S_{n+2} \rar \dar & C_{n,2} \rar \dar & 0 \\
 & C_{n+1,1} \rar \dar & C \rar \dar & 0 \\
 & 0 & 0.
\end{tikzcd}
\]
It follows from (1) that $p^{\frac{c(S)}{p^n}} C_{n,1}=0$ and $p^{\frac{c(S)}{p^{n+1}}} C_{n+1,1}=0$. Then by the exactness of the right-hand column, we have $p^{\frac{c(S)}{p^n}+\frac{c(S)}{p^{n+1}}} C_{n,2}=0$. Inductively, we deduce that for every $m\geq 0$ the cokernel $C_{n,m}=\Cok(R_{n+m}\otimes_{R_n}S_n\to S_{n+m})$ is annihilated by
\[
p^{\frac{c(S)}{p^{n+m-1}}+\cdots+\frac{c(S)}{p^n}} = p^{c(S)\left(\frac{1}{p^{n+m-1}}+\cdots+\frac{1}{p^n}\right)}.
\]
Since
\begin{align*}
\frac{1}{p^{n+m-1}}+\cdots+\frac{1}{p^n} = \frac{1+p+\cdots + p^{m-1}}{p^{n+m-1}} &= \frac{p^m-1}{p^{n+m-1}(p-1)} \\
&= \frac{p^m-1}{p^m}\cdot\frac{p}{p^{n}(p-1)}
\nearrow \frac{p}{p^{n}(p-1)} \quad (m\to\infty),
\end{align*}
we conclude that $\Cok(R_\infty\otimes_{R_n}S_n\to S_\infty)=\varinjlim\limits_{m\geq 0}C_{n,m}$ is annihilated by $p^{\frac{c(S)p}{p^n(p-1)}}$.
\end{proof}

\begin{lemma}[{cf.\ \cite[4.15. Lemma]{An06}}]
\label{lem:4.15}
Let $n\geq N$ be an integer.
  \begin{enumerate}
  \item $S_n^{s.\flat}$ is $(p^\flat)^\varepsilon$-adically complete.
  \item The cokernel of $R_{n+1}^{s.\flat}\otimes_{R_n^{s.\flat}}S_n^{s.\flat} \hookrightarrow S_{n+1}^{s.\flat}$ is annihilated by $(p^\flat)^{\frac{c(S)}{p^n}}$.
  \item The cokernel of $R_\infty^{\flat}\otimes_{R_n^{s.\flat}}S_n^{s.\flat} \to S_\infty^{\flat}$ is annihilated by $(p^\flat)^{\frac{c(S)p}{p^n(p-1)}}$.
  \end{enumerate}
\end{lemma}

\begin{proof}
(1) %Replacing $N$ by $n$, we may assume $n=N$.
Since $(p^\flat)^{p^n\varepsilon} \in \Ker(\Phi^{(n)}_0)$, the assertion follows from \cite[Proposition 3.10 (a)]{INS25}.

(2) For every $m\geq 0$, \cref{cor:3.10} (1) shows that we have an exact sequence
\[
0\to R_{n+m+1}\otimes_{R_{n+m}}S_{n+m} \to S_{n+m+1} \to C_{n+m} \to 0
\]
with $p^{\frac{c(S)}{p^{n+m}}}\cdot C_{n+m}=0$. By reducing modulo $p^\varepsilon$, we have an exact sequence of $R_{n+m+1}/p^\varepsilon R_{n+m+1}$-modules
\[
(R_{n+m+1}/p^\varepsilon R_{n+m+1})\otimes_{R_{n+m}}S_{n+m} \to S_{n+m+1}/p^\varepsilon S_{n+m+1} \to C_{n+m}/p^\varepsilon C_{n+m} \to 0,
\]
which is compatible with Frobenius projections. By taking projective limits, we obtain a complex of $R_n^{s.\flat}$-modules
\[
\varprojlim_m\left((R_{n+m+1}/p^\varepsilon R_{n+m+1})\otimes_{R_{n+m}}S_{n+m}\right) \to S_{n+1}^{s.\flat} \to \varprojlim_m\left(C_{n+m}/p^\varepsilon C_{n+m}\right),
\]
where the right term is annihilated by $(p^\flat)^{\frac{c(S)}{p^n}}$. It remains to show that there exists a canonical isomorphism
\[
R_{n+1}^{s.\flat} \otimes_{R_n^{s.\flat}} S_n^{s.\flat} \cong \varprojlim_m\left((R_{n+m+1}/p^\varepsilon R_{n+m+1})\otimes_{R_{n+m}}S_{n+m}\right).
\]
Since $R_{n+1}^{s.\flat}$ is a finitely generated free $R_n^{s.\flat}$-module, (1) implies that $R_{n+1}^{s.\flat}\otimes_{R_n^{s.\flat}}S_n^{s.\flat}$ is $(p^\flat)^{\varepsilon}$-adically complete. Hence we have canonical maps
%\begin{align*}
%R_{n+1}^{s.\flat}\otimes_{R_n^{s.\flat}}S_n^{s.\flat}
%&\xr{\cong} \varprojlim_m \left(\frac{R_{n+1}^{s.\flat}}{(p^\flat)^{p^m\varepsilon}R_{n+1}^{s.\flat}} \otimes_{\frac{R_{n}^{s.\flat}}{(p^\flat)^{p^m\varepsilon}R_n^{s.\flat}}} \frac{S_n^{s.\flat}}{(p^\flat)^{p^m\varepsilon}}\right) \\
%&\xr{\cong} \varprojlim_m \left(\frac{R_{n+m+1}}{p^\varepsilon R_{n+m+1}} \otimes_{\frac{R_{n+m}}{p^\varepsilon R_{n+m}}} \frac{S_n^{s.\flat}}{(p^\flat)^{p^m\varepsilon}}\right) \\
%&\cong \varprojlim_m \left(R_{n+m+1} \otimes_{R_{n+m}} \frac{S_n^{s.\flat}}{(p^\flat)^{p^m\varepsilon}}\right) 
%\hookrightarrow \varprojlim_m\left(R_{n+m+1} \otimes_{R_{n+m}} \frac{S_{n+m}}{p^\varepsilon S_{n+m}}\right),
%\end{align*}
\begin{align*}
R_{n+1}^{s.\flat}\otimes_{R_n^{s.\flat}}S_n^{s.\flat}
&\xr{\cong} \varprojlim_m \left(R_{n+1}^{s.\flat} \otimes_{R_{n}^{s.\flat}} \frac{S_n^{s.\flat}}{(p^\flat)^{p^m\varepsilon}S_n^{s.\flat}}\right) \\
&\xr{(\ast)} \varprojlim_m \left(R_{n+1}^{s.\flat} \otimes_{R_{n}^{s.\flat}} \frac{S_{n+m}}{p^\varepsilon S_{n+m}}\right) \\
&\xr{\cong} \varprojlim_m \left(\frac{R_{n+1}^{s.\flat}}{(p^\flat)^{p^m\varepsilon}R_{n+1}^{s.\flat}} \otimes_{\frac{R_{n}^{s.\flat}}{(p^\flat)^{p^m\varepsilon}R_n^{s.\flat}}} \frac{S_{n+m}}{p^\varepsilon S_{n+m}}\right) \\
&\xr{\cong} \varprojlim_m \left(\frac{R_{n+m+1}}{p^\varepsilon  R_{n+m+1}}\otimes_{\frac{R_{n+m}}{p^\varepsilon R_{n+m}}} \frac{S_{n+m}}{p^\varepsilon S_{n+m}}\right) \cong \varprojlim_m \left(\frac{R_{n+m+1}}{p^\varepsilon  R_{n+m+1}}\otimes_{R_{n+m}} S_{n+m}\right),
\end{align*}
where the map $(\ast)$ is the map induced by the $m$-th projection $\Phi^{(n)}_{n+m}\colon S_n^{s.\flat}\to S_{n+m}/p^\varepsilon S_{n+m}$. Since $R_{n+1}^{s.\flat}$ is a finitely generated free $R_n^{s.\flat}$-module, we conclude that $(\ast)$ is an isomorphism (see \cref{lem:tensorLim} for more general situation).\footnote{The injectivity of $S_n^{s.\flat}/(p^\flat)^{p^m\varepsilon}S_n^{s.\flat}\to S_{n+m}/p^\varepsilon S_{n+m}$ can be also checked by using the normality of $S_{m+i}$ for $i\geq 0$ (see the proof of \cite[4.5. Proposition (1)]{An06}).}

(3) follows from (2) by an argument similar to that of the proof of \cref{cor:3.10} (2), once we notice the canonical isomorphism
\[
R_\infty^\flat \otimes_{R_n^{s.\flat}} S_n^{s.\flat} \xr{\cong} (R_\infty^{s.\flat}\otimes_{R_n^{s.\flat}}S_n^{s.\flat})^\wedge
\]
obtained by \cref{lem:flat comp} (1), where $(-)^\wedge$ denotes the $(p^\flat)^\varepsilon$-adic completion.
\end{proof}

Now we finish the proof of \cref{thm:Sn perfectoid} by checking (d).

\begin{proof}[Proof of \cref{thm:Sn perfectoid}]
We replace, if necessary, $N$ by a larger number so that $(N+1)\varepsilon \geq c$. Then it is enough to show that the $0$-th projection $\Phi^{(N)}\colon S_N^{s.\flat} \to S_N/p^\varepsilon S_N$ is surjective.
In general, if $f\colon A\to B$ is a ring homomorphism and $I\subseteq A$ is an ideal with $I^nB=(0)$ for some $n>0$, then the surjectivity of the composition $A\xr{f}B\to B/IB$ implies that of $f$.\footnote{Observe that $B \subseteq f(A)+IB \subseteq f(A)+I^2B \subseteq \cdots \subseteq f(A)+I^nB=f(A)$.}
Hence it suffices to show that the composition $S_N^{s.\flat} \xr{\Phi^{(N)}} S_N/p^\varepsilon S_N \to S_N/p^{\varepsilon-\frac{c}{p^N}}S_N$ is surjective. In view of the commutative diagram
\[
\begin{tikzcd}
S_N^{s.\flat} \rar["\Phi^{(N)}"] \dar["\cong","(p^\flat)^{c}"'] & S_N/p^{\varepsilon}S_N \rar & S_N/p^{\varepsilon-\frac{c}{p^N}}S_N \dar["\cong"',"p^{\frac{c}{p^N}}"] \\
(p^\flat)^{c}S_N^{s.\flat} \ar[rr,"\Phi^{(N)}",dashed] \dar[phantom,"\cap"] & & p^{\frac{c}{p^N}}S_N/p^\varepsilon S_N \dar[phantom,"\cap"] \\
S_N^{s.\flat} \ar[rr,"\Phi^{(N)}"'] & & S_N/p^{\varepsilon}S_N,
\end{tikzcd}
\]
we are reduced to checking that the dashed map is surjective. It is enough to prove the following.

\begin{claim}
  \begin{enumroman}
  \item $p^{\frac{c}{p^N}}S_N/p^\varepsilon S_N$ is in the image $\Phi^{(N)}(S_N^{s.\flat})$.
  \item If $x_N\ \mathrm{mod}\ p^\varepsilon S_N\in p^{\frac{c}{p^N}}S_N/p^\varepsilon S_N$ is the image of $x=(x_m\ \mathrm{mod}\ p^\varepsilon S_m)_{m\geq N} \in S_N^{s.\flat}$, then $x\in (p^\flat)^{c}S_N^{s.\flat}$.
  \end{enumroman}
\end{claim}

\begin{pfclaim}
(i) We show that $p^{\frac{c}{p^N}}$ annihilates the cokernel $C\coloneqq \Cok(\Phi^{(N)}\colon S_N^{s.\flat} \to S_N/p^\varepsilon S_N)$. It suffices to prove this after base change via the faithfully flat morphism $R_N^{s.\flat} \to R_\infty^\flat$ (cf.\ \cref{lem:A.7}). By \cref{lem:4.15} (3), the cokernel
\[
C'\coloneqq\Cok(S_N^{s.\flat}\otimes_{R_N^{s.\flat}}R_\infty^\flat \to S_\infty^\flat)
\]
is annihilated by $(p^\flat)^c$.
On the other hand, since
\[
(R_N/p^\varepsilon R_N) \otimes_{R_N^{s.\flat}}R_\infty^\flat
\cong R_N^{s.\flat}/(p^\flat)^\varepsilon R_N^{s.\flat} \otimes_{R_N^{s.\flat}}R_\infty^\flat
\cong R_\infty^{s.\flat}/(p^\flat)^\varepsilon R_\infty^{s.\flat}
\cong R_\infty/p^\varepsilon R_\infty,
\]
we get a map
\begin{align*}
(S_N/p^\varepsilon S_N)\otimes_{R_N^{s.\flat}}R_\infty^\flat &\cong S_N \otimes_{R_N} (R_N/p^\varepsilon R_N) \otimes_{R_N^{s.\flat}}R_\infty^\flat \\
&\cong S_N\otimes_{R_N} (R_\infty/p^\varepsilon R_\infty) \\
&\cong (S_N\otimes_{R_N} R_\infty)\otimes_{R_\infty}(R_\infty/p^\varepsilon R_\infty) \\
&\to S_\infty\otimes_{R_\infty}(R_\infty/p^\varepsilon R_\infty) 
\cong S_\infty/p^\varepsilon S_\infty.
\end{align*}

These identifications fit into the commutative diagram with exact rows
\[
\begin{tikzcd}
0 \rar & S_N^{s.\flat} \otimes_{R_N^{s.\flat}} R_\infty^\flat \rar \dar["\Phi^{(N)}\otimes 1"'] & S_\infty^\flat \rar \dar["\Phi^{(N)}"'] & C' \rar \dar & 0 \\
0 \rar & (S_N/p^\varepsilon S_N)\otimes_{R_N^{s.\flat}}R_\infty^\flat \rar & S_\infty/p^\varepsilon S_\infty \rar & C'' \rar & 0,
\end{tikzcd}
\]
where we define $C''$ as the cokernel. Since the middle vertical arrow is surjective by Witt-perfect almost purity theorem \cite[Theorem 5.9]{NS18}, the snake lemma shows that we have a surjection $\Ker(C'\to C'') \twoheadrightarrow C\otimes_{R_N^{s.\flat}}R_\infty^\flat$. Hence we have the implication
\[
(p^\flat)^c\cdot\Ker(C'\to C'')=0 \Longrightarrow p^{\frac{c}{p^N}}\cdot(C\otimes_{R_N^{s.\flat}}R_\infty^\flat) = 0.
\]
The former holds because $(p^\flat)^c\cdot C'=0$.

(ii) Take $m\geq N$, and let $y_m\coloneqq x_m/p^{\frac{c}{p^m}} \in \Frac(S_m)$. Since $x_N\in p^{\frac{c}{p^N}}S_N$, we have
\[
y_m^{p^m} = \frac{x_m^{p^m}}{p^{c}} \in \frac{x_N^{p^N}}{p^{c}} + p^{(N+1)\varepsilon -c}S_N \subset S_N \subset S_m.
\]
Since $S_m$ is normal, we get $y_m\in S_m$. Hence $x/(p^\flat)^{c} = (y_m\ \mathrm{mod}\ p^\varepsilon S_m)_{m\geq N} \in S_N^{s.\flat}$, and so $x \in (p^\flat)^{c}S_N^{s.\flat}$.
\end{pfclaim}

This completes the proof of \cref{thm:Sn perfectoid}.
\end{proof}

%%%%%%%%%%%%%%%%%%%%%%%%%%%%%%%%%%%%%%%%%%%%%%%%%%%%%%%
%%%%%%%%%%%%%%%%%%%%%%%%%%%%%%%%%%%%%%%%%%%%%%%%%%%%%%%
\appendix
\def\thesection{\Alph{section}}
\section{Flatness of completions}
\label{s:A}

It is well-known that if $R$ is a Noetherian ring with an ideal $I\subset R$, then the $I$-adically completion $R\to\wh{R}$ is flat. In this appendix, we discuss such a flatness for more general rings (especially, the $p$-adically completion $R_\infty\to\wh{R_\infty}$ of the colimit of the perfectoid tower associated to a complete regular local ring). The main idea is based on \cite[A.7. Lemma]{An06}.
In order to develop the theory in the most general setting, we consider the following condition for a pair $(A,I)$, introduced in \cite[Chapter 0, \S7.4.(c)]{FKI}.

  \begin{enumerate}
  \item[\textbf{(APf)}] For any finitely generated $A$-module $M$ and any finitely generated $A$-submodule $N\subseteq M$, the filtration $\{N\cap I^nM\}_{n\geq 0}$ on $N$ induced by the $I$-adic filtration of $M$ is $I$-adic.
  \end{enumerate}

Due to the Artin--Rees lemma, if $A$ is Noetherian, then it satisfies \textbf{(APf)} for any ideal $I\subset A$. Furthermore, if $V$ is an $a$-adically separated valuation ring for an element $a\in\mathfrak{m}_V\setminus\{0\}$, then the pair $(V,(a))$ satisfies \textbf{(APf)} (see \cite[Chapter 0, Example 8.5.14]{FKI}).

\begin{lemma}
\label{lem:flat comp}
Let $A$ be a ring and $I\subseteq A$ a finitely generated ideal. Suppose $A$ is $I$-adically complete and satisfies \emph{\textbf{(APf)}}.
Let $F$ be a flat (resp.\ faithfully flat) $A$-module. Let $(-)^\wedge$ denote the $I$-adically completion.
  \begin{enumerate}
  \item For any finitely presented $A$-module $M$, the canonical map
  \[
  M\otimes_A\wh{F} \to (M\otimes_AF)^\wedge
  \]
  is an isomorphism.
  \item $\wh{F}$ is a flat (resp.\ faithfully flat) $A$-module.
  \end{enumerate}
\end{lemma}

\begin{proof}
(1) Take a finite presentation
\[
A^{\oplus m} \to A^{\oplus n} \to M\to 0,
\]
and consider the commutative diagram with exact rows
\[
\begin{tikzcd}
A^{\oplus m}\otimes_A\wh{F} \rar \dar & A^{\oplus n}\otimes_A\wh{F} \rar \dar & M\otimes_A\wh{F} \rar \dar & 0 \\
(A^{\oplus m}\otimes_AF)^\wedge \rar & (A^{\oplus n}\otimes_AF)^\wedge \rar & (M\otimes_AF)^\wedge \rar & 0.
\end{tikzcd}
\]
Here the exactness of the second row is a consequence of \textbf{(APf)} (see \cite[Chapter 0, Proposition 7.4.10]{FKI}). Since the first two vertical arrows are clearly isomorphisms, the third is also an isomorphism.

(2) Since any $A$-module is isomorphic to an inductive limit of finitely presented $A$-modules, it suffices to show that the functor $-\otimes_A\wh{F}$ on the category of finitely presented $A$-modules is exact (resp.\ faithful exact). In view of (2), this amounts to showing that the functor $(-\otimes_AF)^\wedge$ is exact (resp.\ faithful exact), which follows from the assumption of $F$ and \textbf{(APf)}.
\end{proof}

\begin{lemma}
\label{lem:flat injlim}
Let $\{f_i\}_{i\in I}\colon \{A_i\}_{i\in I}\to \{B_i\}_{i\in I}$ be a morphism of inductive systems of rings indexed by a directed set $I$. If each $f_i$ is flat (resp.\ faithfully flat), then so is $\varinjlim\limits_{i\in I}f_i\colon \varinjlim\limits_{i\in I}A_i\to \varinjlim\limits_{i\in I}B_i$.
\end{lemma}

\begin{proof}
Since $\Cok(\varinjlim\limits_{i\in I}f_i)=\varinjlim\limits_{i\in I}\Cok(f_i)$, it is enough to prove the assertion only in the ``flat'' case, which is easy; see \cite[Chapter I, \S2.7, Proposition 9]{BouAC}.
\end{proof}

\begin{lemma}[{cf.\ \cite[A.7. Lemma]{An06}}]
\label{lem:A.7}
Let $R$ be a ring and $I\subseteq R$ a finitely generated ideal.
Let $\{R_i,t_i\}_{i\geq 0}$ be a tower of rings with $R_0=R$.
Suppose that for every $i\geq 0$, $R_i$ is $I$-adically complete and satisfies \emph{\textbf{(APf)}}.
If $R_i\to R_\infty$ is flat (resp.\ faithfully flat) for every $i\geq 0$, then so is $R_\infty\to\wh{R_\infty}$.
\end{lemma}

\begin{proof}
By \cref{lem:flat comp} (2), $R_i\to \wh{R_\infty}$ is flat (resp.\ faithfully flat) for every $i\geq 0$. But then so is $R_\infty\to\wh{R_\infty}$ by \cref{lem:flat injlim}.
\end{proof}

\begin{remark}
If $R_\infty$ is $I$-adically Zariskian, then $R_\infty\to\wh{R_\infty}$ is flat if and only if it is faithfully flat (\cite[Chapter 0, Proposition 7.3.8]{FKI}).
\end{remark}

Finally, let us include the following lemma concerning the case where tensor products commute with projective limits.

\begin{lemma}[{cf.\ \cite[\href{https://math.stackexchange.com/questions/181004/inverse-limit-of-modules-and-tensor-product}{q181004}]{MSE}}]
\label{lem:tensorLim}
Let $A$ be a ring. Let $M$ be a finitely presented $A$-module, and $\{N_i\}_{i\in I}$ a strict\footnote{A projective system $\{X_i,f_{ij}\colon X_j\to X_i\}_{i\in I}$ of objects in a category indexed by a directed set $I$ is said to be \emph{strict} if all transition maps $f_{ij}$ for $i\leq j$ are epimorphic.} projective system of $A$-modules indexed by a directed set $I$. If $\Tor^A_1(M,N_i)=0$ for any $i\in I$ (e.g., $M$ or every $N_i$ is flat), then the canonical map
\[
M\otimes_A (\varprojlim_{i\in I}N_i) \to \varprojlim_{i\in I}(M\otimes_AN_i)
\]
is an isomorphism.
\end{lemma}

\begin{proof}
Take a finite presentation
\[
A^{\oplus m}\xr{\varphi} A^{\oplus n}\xr{\pi} M\to 0,
\]
and consider the commutative diagram with exact rows
\[
\begin{tikzcd}
A^{\oplus m}\otimes_A(\varprojlim\limits_{i\in I}N_i) \rar \dar & A^{\oplus n}\otimes_A(\varprojlim\limits_{i\in I}N_i) \rar \dar & M\otimes_A(\varprojlim\limits_{i\in I}N_i) \rar \dar & 0 \\
\varprojlim\limits_{i\in I}(A^{\oplus m}\otimes_AN_i) \rar & \varprojlim\limits_{I\in I}(A^{\oplus n}\otimes_AN_i) \rar & \varprojlim\limits_{i\in I}(M\otimes_AN_i) \rar & 0.
\end{tikzcd}
\]
Here the exactness of the second row follows from the vanishing $\Tor_1^A(M,N_i)=0$ and the fact that the projective systems $\{(\Ker\pi)\otimes_AN_i\}_{i\in I}$ and $\{(\Ker\varphi)\otimes_AN_i\}_{i\in I}$ are again strict. Since the first two vertical maps are clearly isomorphisms, the third is also an isomorphism.
\end{proof}

%%%%%%%%%%%%%%%%%%%%%%%%%%%%%%%%%%%%%%%%%%%%%%%%%%%%%%%%%%%%%%%%%%%%%%%
%%%%%%%%%%%%%%%%%%%%%%%%%%%%%%%%%%%%%%%%%%%%%%%%%%%%%%%%%%%%%%%%%%%%%%%
%%%%%%%%%%%%%%%%%%%%%%%%%%%%%%%%%%%%%%%%%%%%%%%%%%%%%%%%%%%%%%%%%%%%%%%

\end{document}